\documentclass[10pt,paper=letter]{scrartcl}

\usepackage{graphicx}
\usepackage{amsmath}
\usepackage{amsfonts}
\usepackage{amssymb}
\usepackage{bm}

\usepackage[numbers,sort]{natbib}
\usepackage{xspace}
\usepackage{xcolor}
\usepackage{amssymb}
\usepackage{enumerate}
\usepackage{verbatim}
\usepackage{subfigure}
\usepackage{bbm}
\usepackage{xr}
\usepackage{threeparttable}

\usepackage[colorlinks,linkcolor=blue,citecolor=red]{hyperref}
\usepackage{hypernat}
\usepackage[bf,format=plain,singlelinecheck=false]{caption}

\usepackage{algorithm}
\usepackage{algorithmic}

\newcommand\define{\mathrel{\overset{\makebox[0pt]{\mbox{\normalfont\tiny\sffamily def}}}{=}}}
\newcommand\bbeta{\boldsymbol{\beta}}

\newcommand{\err}{e}  
\newcommand{\st}{\mbox{ s.t. }}
\newcommand{\bx}{\mathbf{x}}
\newcommand{\bh}{\mathbf{h}}

\newcommand{\bw}{\mathbf{w}}
\newcommand{\bv}{\mathbf{v}}

\newcommand{\pen}{\textnormal{pen}}
\newcommand{\ATEsig}{\sigma_{\pen}}

\newcommand{\norm}[1]{\left\Vert #1 \right\Vert}
\newcommand{\shortnorm}[1]{\Vert #1 \Vert}

\newcommand{\ACE}{\textnormal{ACE}}
\newcommand{\unadj}{\textnormal{unadj}}
\newcommand{\adj}{\textnormal{adj}}
\newcommand{\ACEhat}{\hat{\tau}}
\newcommand{\OLS}{\textnormal{OLS}}
\newcommand{\Lasso}{\textnormal{Lasso}}
\newcommand{\Ridge}{\textnormal{Ridge}}
\newcommand{\naiveEN}{\textnormal{naiveEN}}
\newcommand{\EN}{\textnormal{EN}}
\newcommand{\ada}{\textnormal{ada}}
\newcommand{\ATEunadj}{\sigma_{\unadj}}

\newtheorem{thm}{Theorem}
\newtheorem{lem}{Lemma}

\newtheorem{condition}{Condition}
\newtheorem{remark}{Remark}

\newtheorem{proof}{Proof}

\DeclareMathOperator*{\argmin}{arg\,min}

\title{Penalized regression adjusted causal effect estimates in high dimensional randomized experiments}

\author{
	Hanzhong Liu\thanks{Center for Statistical Science and Department of Industrial Engineering, Tsinghua University, Beijing, 100084, China.}
	\and
	Yuehan Yang\thanks{School of Statistics and Mathematics, Central University of Finance and Economics, Beijing, 102206, China.}
}

\begin{document}
	
\maketitle

\begin{abstract}
  \noindent	\footnotesize\textbf{Abstract}

  \noindent	 Regression adjustments are often considered by investigators to improve the estimation efficiency of causal effect in randomized experiments when there exists many pre-experiment covariates. In this paper, we provide conditions that guarantee the penalized regression including the Ridge, Elastic Net and Adapive Lasso adjusted causal effect estimators are asymptotic normal and we show that their asymptotic variances are no greater than that of the simple difference-in-means estimator, as long as the penalized estimators are risk consistent. We also provide  conservative estimators for the asymptotic variance which can be used to construct asymptotically conservative confidence intervals for the average causal effect (ACE). Our results are obtained under the Neyman-Rubin potential outcomes model of randomized experiment when the number of covariates is large. Simulation study shows the advantages of the penalized regression adjusted ACE estimators over the difference-in-means estimator.
  \ \\

  \noindent \textbf{Keywords}

  \noindent randomized experiment, Neyman-Rubin model, average treatment effect, Lasso, Ridge, Elastic Net, Adaptive Lasso, asymptotic normality, causal inference
\end{abstract}

\section{Introduction}

Randomized experiments are golden standard for conducting inference for various causal effects. Randomization ensures that treatment assignment does not depend on any potential confounders which makes many causal effects can be identified without bias. The Neyman-Rubin potential outcomes model ~\cite{Neyman:1923,Rubin:1974} are often used to analyze randomized experiments. This model assumes that, for each unit $i$ in the experiment, there exists a pair of potential outcomes representing the responses of the unit $i$ receiving or not receiving the treatment (for example, a training program or a new drug). They are denoted as $a_i$ if the unit $i$ is assigned to the treatment group and $b_i$ if it is exposed to the control group. Both $a_i$ and $b_i$ are assumed to be fixed, nonrandom quantities and all the units in the experiment are thought of as the population. The source of randomness comes only from the assignment of treatment controlled by the experimenter. The causal effect for the unit $i$ can be defined as the difference of the two potential outcomes $a_i - b_i$.  Since each unit is either exposed to the treatment group or to the control group, we cannot observe $a_i$ and $b_i$ simultaneously. Hence, the unit level causal effect is not identifiable without further assumptions. Fortunately, the Average Causal Effect (ACE) defined by the averages of $a_i - b_i$ over all the units in the experiment, $\sum_{i=1}^{n} \left( a_i - b_i \right) /n $ where $n$ is the total number of units, can be estimated without bias by the difference of average outcomes in the treatment and control groups if the Stable Unit Treatment Value Assumption (SUTVA) holds~\cite{Neyman:1923}. Basically, SUTVA states that there is only one version of treatment, and that the potential outcome of one unit should be unaffected by the particular assignment of treatments to the other units. This simple estimator is often called the difference-in-means estimator or the unadjusted estimator.

In practice, randomized experiments often occur with pre-experiment covariate information is collected about each unit, such as demographic characteristics or behavior history. Some of the covariates are relevant to the outcomes and others may not. Taking the relevant covariates into account when analyzing the experimental outcomes may help to improve the estimation accuracy of the ACE. This motivate investigators to consider linear regression adjustment when estimating the ACE.  In low-dimension scenario where the number of covariates $p$ is fixed when the total number of units $n$ goes to infinity, it was shown that if the treatment by covariates interactions are included in the adjustment, the resulting ordinary least squares adjusted ACE estimator is asymptotically unbiased and normal with asymptotic variance no greater than that of the difference-in-means estimator~\cite{lin2013}. In high-dimension scenario where $p$ is comparable to or even much larger than $n$, the ordinary least squares does not work due to overfitting. Since many covariates may be irrelevant to the potential outcomes, penalization or some form of regularization is required to select important covariates and to make effective regression adjustment. The low-dimensional theoretical results have been extended to high-dimensional setting by replacing the ordinary least squares by the Lasso~\cite{bloniarz2015lasso,Wager2016}.

The Lasso~\cite{Tibshirani1996} is an $l_1$ penalized method which is population among many other because it can perform parameter estimation and model selection simultaneously and it is computational feasible. However, the Lasso has some drawbacks: (1) it selects at most $n$ covariates; (2) its estimation and prediction errors can be large when the correlations between covariates are high; (3) it is not guaranteed to be both model selection consistent and asymptotically normal for one particular value of the regularization parameter. In order to overcome these drawbacks and deal with different sparsity structures of the regression coeffients, various variants of the Lasso have been proposed in high-dimensional sparse linear regression model, such as the Elastic Net~\cite{Zou2005}, adaptive Lasso~\cite{Zou2006,HuangZhang2008}, group Lasso~\cite{Yuan2006}, SCAD~\cite{FanLi2001}, Lasso+OLS~\cite{liu2013} among many others. Detailed discussions about the properties of the Lasso and its variants can be found in a series of literatures, including~\cite{KnightFu2000,ZhaoYu2006,Fan2008,Huang2008,Bickel2009,unifiedpaper2009journal}. Most of the results are obtained under the linear regression model assuming certain type of sparsity.

The Neyman-Rubin potential outcomes model is different from linear regression model in the following aspects: (1) it makes no assumptions of linearity on the relationship between the potential outcomes and the covariates; (2) both the potential outcomes and the covariates are assumed to be fix quantities; (3) the treatment group is sampled without replacement from a finite population so that the observations are not independent; (4) the only randomness comes from the random assignment of treatment instead of an additive random error term. Although we use linear regression to analyze the outcomes of randomized experiments, but the linear regression model is misspecified. Understanding the properties of the penalized regression adjusted ACE estimators in the Neyman-Rubin model remains insufficient in the literature. In this paper, we fill in this gap by investigating the theoretical properties of the Ridge, Elastic Net, and Adaptive Lasso adjusted ACE estimator in the Neyman-Rubin model.

\textbf{ Our contributions} are summarized as follows:
\begin{enumerate}
\item We establish a uniform theory on the asymptotic normality of penalized regression adjusted ACE estimates and show that their asymptotic variances are less than or equal to that of the difference-in-means estimator under appropriate conditions. Our results are obtained under the Neyman-Rubin non-parametric model of randomization and under the finite population asymptotic framework. It supplements the current results that are obtained under a super population framework, assuming that units are independent and identically distributed which are drawn from a hypothetical population; see \cite{Belloni2013,Belloni2013program,Wager2016} for example. 
\item We apply the uniform theory to the Ridge, Elastic Net and Adaptive Lasso adjusted ACE estimators and provide conservative estimators for the asymptotic variance based on the residuals, yielding asymptotically conservative confidence intervals for the ACE which can be comparable to interval constructed by the Lasso adjusted ACE estimator and are substantially narrower than that constructed by the difference-in-means estimator. As a by-product, we extend the $l_1$ consistency of the Ridge, Elastic Net and Adaptive Lasso from linear regression model to high dimensional Neyman-Rubin causal inference model \cite{Bunea2008,HuangZhang2008}.
\item We conduct simulation studies to compare the finite sample performance of the Lasso, Ridge, Elastic Net and Adaptive Lasso adjusted ACE estimators with the difference-in-means estimator and show the advantages of the penalized regression adjustments.
\end{enumerate}

The rest of the paper is organized as follows. Section 2 presents the framework of average causal effect estimation with regression adjustment and Section 3 shows the theoretical properties of penalized regression adjusted ACE estimators including the Ridge, Elastic Net and Adaptive Lasso. Section 4 contains numerical results and Section 5 summarizes the paper. All the proofs and tables are in the Appendix.

\section{Causal effect estimation with regression adjustment}
\subsection{Framework and notations}

In the section, we introduce notations of the Neyman-Rubin potential outcomes model for randomized experiments; see \cite{Neyman:1923}, \cite{Rubin:1974}, and \cite{holland1986statistics} for more details. We follow the notations introduced in \cite{freedman2008regression_a}, \cite{lin2013} and \cite{bloniarz2015lasso}.  For each unit $i$ in the experiment, we have denoted its potential outcomes under treatment and control as $a_i$ and $b_i$ respectively. We will use a random indicator $T_i$ to denote the assigned treatment status, which takes on a value $1$ for a treated unit, or $0$ for an untreated (or control) unit.  We assume that the set of treated units is sampled without replacement from the full population (all the units in the experiment), where the size of the treatment group and control group, $n_A$ and $n_B$, are fixed before the experiment; thus the $T_i$ are identically distributed but not independent.
We denote $Y_i$ as the observed outcome for unit $i$, thus,
\begin{equation*}
Y_i = T_i a_i + (1-T_i) b_i.
\end{equation*}
This equation implies that the experimenter observes the potential outcome $a$ for those who is assigned to the treatment group, and the potential outcome $b$ for those who is not. Note that the model does not make any assumption on the relationship between the observed outcomes and the pre-experiment, baseline covariates about the units in the study, for example, age, gender and genetic makeup. We denote the covariates for unit $i$ as a column vector $\bx_i = (x_{i1},...,x_{ip})^T \in \mathbb{R}^{p}$ and the full design matrix of the experiment as $X=(\bx_1,...,\bx_n)^T$. In the following, we will assume that the covariates are centered at their population means, that is, $n^{-1} \sum_{i=1}^{n} \bx_i = \mathbf{0}$.

Let $A = \{i \in \{1,...,n\}:  T_i = 1\}$ be the set of treated units, and similarly define
the set of control units as $B$. Denote the number of treated and control units as  $n_A=\left|A\right|$ and $n_B=\left|B\right|$, respectively, so that $n_A+n_B=n$. We introduce the notation $\bar{\cdot}_A $ and $\bar{\cdot}_B  $ to indicate averages of quantities over treated and control units respectively. For example, the average value of the potential outcomes and the covariates in the treatment group are denoted as
\[
\bar{a}_A =  n_A^{-1}\hbox{$\sum_{i \in A}$} a_i, \ \bar{\bx}_A  = n_A^{-1} \hbox{$\sum_{i \in A}$}\bx_i,
\]
respectively. Note that these are random quantities since the set $A$ is determined by the random treatment assignment. We will add a bar above an quantity to denote its average over the whole population, such as
\[
\bar{a} = n^{-1} \hbox{$\sum_{i=1}^{n}$} a_i, \ \bar{b}   = n^{-1} \hbox{$\sum_{i=1}^{n}$}b_i.
\]

Note that the averages of potential outcomes over the whole population are not considered random in this model, but they are unobservable. The parameter we are interested to infer is the average causal effect (ACE), that is, the average effect of the treatment over the whole population in the study defined by
\[
\ACE  = n^{-1} \hbox{$\sum_{i=1}^n$} ( a_i - b_i) = \bar{a} - \bar{b}  \triangleq \tau.
\]
The ACE, denoted by $\tau$, represents the difference between the average outcomes if everyone had received the treatment (for example, a training program), and the average outcomes if no one had received it. Our main goal is to outline conditions which guarantee efficiency gain of the ACE estimates through penalized regression adjustments including the Ridge, Elastic Net and Adaptive Lasso.

\subsection{Causal effect estimation with regression adjustment in low-dimension}

The ACE can be naturally estimated by a plug-in method, that is, replacing the population averages with the sample averages:
\begin{equation*}
\label{unadj}
\ACEhat_\unadj \define \bar{a}_A - \bar{b}_B,
\end{equation*}
This naive estimator is often called the difference-in-means estimator and it does not cooperate with covariates information. We use the subscript ``unadj" to indicate an estimator without regression adjustment. The foundational work in \cite{Neyman:1923} and the work in \cite{freedman2008regression_a} showed that, under a randomized assignment of treatment and certain moment conditions, $\ACEhat_\unadj$ is unbiased and asymptotically normal for $\ACE$ and a conservative procedure was proposed to estimate its variance.

However, the unadjusted estimator is not optimal in the sense that its asymptotic variance can be reduced by regression adjustment. A natural way of performing regression adjustment in low-dimension scenario is to regress the observed outcome $Y_i$ on the treatment indicator $T_i$ and the covariates $\bx_i$ with intercept, and the Ordinary Least Squares (OLS) coefficient on $T_i$, denoted by $\ACEhat_{\OLS}$,  is an estimator of ACE.  Freedman \cite{freedman2008regression_a,freedman2008randomization} critiqued the usages of $\ACEhat_{\OLS}$ by showing that, it is asymptotically unbiased and asymptotically normal, but its asymptotic variance could be larger than that of the unadjusted difference-in-means estimator. One reason is that the population OLS coefficient of regressing $a_i$ on $\bx_i$ and that of regressing $b_i$ on $\bx_i$ could be different which could not be recognized by one sample version OLS regression coefficient estimate.

The above discussion motivates to consider incorporating the treatment by covariates interaction $T_i \cdot \bx_i $ into the regression adjustment~\cite{lin2013}. As pointed out by~\cite{lin2013}, we can run a linear regression of $Y_i$ on $T_i$, the covariates $\bx_i$ and their interactions $T_i \cdot \bx_i $, and use the OLS estimate of the coefficient on $T_i$ as an estimator for ACE. This is equivalent to perform two linear regressions: regress $Y_i$ on $\bx_i$ in  the treatment and control group separately. We denote the two OLS coefficients on the covariates as $\hat\bbeta^{(a)}_{\OLS}$ and $\hat\bbeta^{(b)}_{\OLS}$. According to~\cite{lin2013}, the OLS coefficient of $T_i$ in the fully interacted regression for $Y_i$ is
\begin{align*}
\ACEhat_\adj =
\left[ \bar{a}_A - \left( \bar{\bx}_A  \right)^T \hat \bbeta^{(a)}_{\OLS} \right]  - \left[ \bar{b}_B - \left( \bar{\bx}_B  \right)^T \hat \bbeta^{(b)}_{\OLS} \right].
\end{align*}

The terms $ \bar{\bx}_A  $ and $ \bar{\bx}_B  $ represent the variation of the covariates in the treatment and control group relative to the full population, and the adjustment vectors adjust for these differences\footnote{Note that, we assume the population means of the covariates are zero, $\bar{\bx}=0$.}. The estimator $\ACEhat_\adj $ has some finite-sample bias, but \cite{lin2013} pointed out that it decays quickly at the rate of $1/n$ under fourth moment conditions on the potential outcomes and covariates when $p$ is fixed. Moreover, its asymptotic variance is never higher than the unadjusted estimator, and asymptotically conservative confidence interval for the true ACE can be constructed. In practice, the adjusted vectors $\hat\bbeta^{(a)}_{\OLS}$ and $\hat\bbeta^{(b)}_{\OLS}$ can be replaced by other estimation methods such as the Ridge, Lasso, Elastic Net and Adaptive Lasso when the number of covariates is large.

\subsection{Penalized regression adjustment in high-dimension}

In modern randomized experiments, the number of covariates can be very large, even larger than the sample size, especially when interactions between covariates are taken into account or using non-parametric smooth-splines. For example, in a clinical trial, the demographic and genetic information may be recorded about each patents; in tech industry, a huge mount of behavior data could be collected about each user. In this high-dimensional setting, the OLS estimator performs poorly due to overfitting. Most of the time, not all the covariates are relevant to the outcomes under consideration. Variable selection or penalized regression methods, such as the Ridge, Lasso and its variants, are useful for adjustment purpose. For example, we can use the Lasso to estimate the adjustment vector~\cite{bloniarz2015lasso}. In particular, let $\hat{\bbeta}_{\textnormal{Lasso}}^{(a)}$ and $\hat{\bbeta}_{\textnormal{Lasso}}^{(b)} $ be the Lasso estimates of the coefficients when regressing $Y_i$ on $\bx_i$ in the treatment and control group, respectively,
\begin{equation}
\begin{split}
\label{def-lasso-a}
\hat{\bbeta}_{\textnormal{Lasso}}^{(a)} = \argmin_{\bbeta} \biggl[
\frac{1}{2n_A} &\sum_{i \in A} \left( a_i - \bar{a}_A - (\bx_i -  \bar{\bx}_A )^T \bbeta \right)^2  + \lambda_a \sum_{j=1}^{p} |\beta_j| \biggr],
\end{split}
\end{equation}
\begin{equation}
\begin{split}
\label{def-lasso-b}
\hat{\bbeta}_{\textnormal{Lasso}}^{(b)} = \argmin_{\bbeta} \biggl[
\frac{1}{2n_B} &\sum_{i \in B} \left( b_i - \bar{b}_B - (\bx_i -  \bar{\bx}_B )^T \bbeta \right)^2 + \lambda_b \sum_{j=1}^{p} |\beta_j| \biggr],
\end{split}
\end{equation}
where $\lambda_a$ and $\lambda_b$ are regularization parameters for the Lasso which could be chosen by cross-validation. The Lasso adjusted ATE estimator is defined by
\begin{equation}
\begin{split}
\label{def-ate-lasso}
\hat{\tau}_{\Lasso} = &
\left[ \bar{a}_A - \left( \bar{\bx}_A - \bar{\bx} \right)^T\hat\bbeta^{(a)}_{\textnormal{Lasso}} \right]
 - \left[ \bar{b}_B - \left( \bar{\bx}_B - \bar{\bx} \right)^T\hat\bbeta^{(b)}_{\textnormal{Lasso}} \right].  \nonumber
\end{split} \nonumber
\end{equation}
Under the Neyman-Rubin model, and under fourth moment conditions on the potential outcomes and covariates and other appropriate conditions that guarantee the $l_1$ consistency of the Lasso, it was shown by \cite{bloniarz2015lasso} that $\ACEhat_{\Lasso} $ enjoys similar asymptotic and finite-sample advantages as $\ACEhat_{\OLS}$.

As aforementioned, the adjusted vector could be any penalized regression estimators beyond the Lasso. Let $\hat{\bbeta}_{\pen}^{(a)}$ and $\hat{\bbeta}_{\pen}^{(b)} $ be some penalized regression estimators,
\begin{equation}
\begin{split}
\label{def-pen-a}
\hat{\bbeta}_{\pen}^{(a)} = \argmin_{\bbeta} \biggl[
\frac{1}{2n_A} &\sum_{i \in A} \left( a_i - \bar{a}_A - (\bx_i -  \bar{\bx}_A )^T \bbeta \right)^2  + \sum_{j=1}^{p} p_{a}(\beta_j)  \biggr],
\end{split}
\end{equation}
\begin{equation}
\begin{split}
\label{def-pen-b}
\hat{\bbeta}_{\pen}^{(b)} = \argmin_{\bbeta} \biggl[
\frac{1}{2n_B} &\sum_{i \in B} \left( b_i - \bar{b}_B - (\bx_i -  \bar{\bx}_B )^T \bbeta \right)^2 + \sum_{j=1}^{p}  p_{b}(\beta_j) \biggr],
\end{split}
\end{equation}
where $p_a(t)$ and $p_b(t)$ are functions of penalty, such as, the Lasso:
$$p_a(t) =  \lambda_a  | t | ,$$
the Elastic Net:
$$p_a(t) =  \lambda_{a,1}  | t | + (1/2) \lambda_{a,2} t^2 ,$$ 
the SCAD:
$$
p_a(t) = \lambda_a \left\{ I(t\leq \lambda_a) + \frac{ (c \lambda_a - t)_{+} }{ (c-1)\lambda_a } I(t > \lambda_a) \right\}, \quad \textnormal{for some} \ c > 2,
$$
among many others. Similar arguments can be applied to $p_b(t)$. The general penalized regression adjusted ACE estimator is defined by
\begin{equation}
\begin{split}
\label{def-ate-pen}
\hat{\tau}_{\pen} = &
\left[ \bar{a}_A - \left( \bar{\bx}_A  \right)^T\hat\bbeta^{(a)}_{\pen} \right]
 - \left[ \bar{b}_B - \left( \bar{\bx}_B  \right)^T\hat\bbeta^{(b)}_{\pen} \right].  \nonumber
\end{split} \nonumber
\end{equation}

In the next section, we analyze this estimator under the Neyman-Rubin model and the finite population asymptotic framework, and provide conditions under which $\hat{\tau}_{\pen}$  is asymptotically normal and asymptotically more efficient than the unadjusted difference-in-means estimator.

\section{Theoretical results}\label{sec:theory}



\subsection{Notations}
We follow the notations in \cite{bloniarz2015lasso}. For a $p$-dimensional vector $\bbeta \in R^p$ and a subset $S \subset \{1,...,p\}$, denote $\beta_j$ the $j$-th component of $\bbeta$, $\bbeta_S = (\beta_j: j \in S)^T$,
$S^c$ the complement of $S$, and $|S|$ the cardinality of the set $S$. For column vectors $\mathbf{u}=(u_1,...,u_m)^T$, $\mathbf{v}=(v_1,...,v_m)^T$, let $\|\mathbf{u}\|_2 = \sqrt{ \sum_{i=1}^{m} u_i^2 }$, $\|\mathbf{u}\|_1 = \sum_{i=1}^{m} |u_i|$, $\|\mathbf{u}\|_\infty = \max_{i=1,\ldots,m} |u_i|$ and $\|\mathbf{u}\|_0 = | \{j: u_j \neq 0 \} |$ be the $l_2$, $l_1$, $l_{\infty}$ and $l_0$ norm, respectively. Denote $\sigma^2_{\mathbf{u}} = 1/(m-1) \sum_{i=1}^{m} (u_i - \bar{\mathbf{u}})^2$ be the variance of $\mathbf{u}$ and let $\sigma_{\mathbf{u},\mathbf{v}}$ be the covariance of $\mathbf{u}$ and $\mathbf{v}$. For a given $m\times m$ matrix $D$, let $\lambda_{\textnormal{min}} (D)$ and $\lambda_{\textnormal{max}} (D)$ be the smallest and largest eigenvalues of $D$
respectively, and let $D^{-1}$ be the inverse of the matrix $D$. Let $\stackrel{d}{\rightarrow}$ and $\stackrel{p}{\rightarrow}$ stand for convergence in distribution and in probability, respectively.

\subsection{Decomposition of the potential outcomes}
Similar to the theoretical study of the Lasso adjusted ACE estimator in \cite{bloniarz2015lasso}, we will also decompose the potential outcomes onto the space spanned by sparse linear combinations of the relevant covariates, that is, for $i = 1, \cdots, n$, let
\begin{equation}
\label{decom-a}
a_i = \bar{a} + ( \bx_i  )^T \bbeta^{(a)} + \err^{(a)}_i ,
\end{equation}
\begin{equation}
\label{decom-b}
b_i = \bar{b} + ( \bx_i )^T \bbeta^{(b)} + \err^{(b)}_i,
\end{equation}
where $ \bbeta^{(a)}$ and $ \bbeta^{(a)}$ are the projection coefficients which are the OLS coefficients of regressing population vectors of $a_i$'s and $b_i$'s on the relevant covariates in $\bx_i$; $\err^{(a)}_i$ and $\err^{(b)}_i$ are zero-mean approximation errors.  The projection coefficients may have different sparsity structure for different randomized experiments. Note that all the quantities in the decomposition are fixed, deterministic numbers, thus we have not assumed linear regression models. In order to pursue a theory for the penalized regression adjusted ACE estimator, we will add assumptions on the populations of $a_i$, $b_i$, $\bx_i$, the approximation errors $\err^{(a)}_i$ and $\err^{(b)}_i$, and the sparsity of $ \bbeta^{(a)}$ and $ \bbeta^{(b)}$.

\subsection{Asymptotic normality of penalized regression adjusted ACE estimator}

We require the following assumptions on the potential outcomes and covariates. The first set of assumptions (\ref{cond:stability}-\ref{cond:limit}) are similar to those assumed in \cite{lin2013} and \cite{bloniarz2015lasso}.
\begin{condition}\label{cond:stability}
Stability of treatment assignment probability. For some $p_A \in \left(0,1\right)$.
\begin{eqnarray}
\label{cond:treat-prob}
n_A/n 
\rightarrow p_A, \ \textnormal{as} \ n \rightarrow \infty.
\end{eqnarray}
\end{condition}

\begin{condition} \label{cond:moment}
The centered moment conditions.
There exists a fixed constant $L>0$ such that, for all $n=1,2,...$ and $j=1,...,p$,
\begin{equation}
\label{cond:xmoment}
 \hbox{$n^{-1} \sum_{i=1}^{n}$} \left( x_{ij} \right)^4 \leq L;
\end{equation}
\begin{equation}
\label{cond:errmoment}
  \hbox{$n^{-1} \sum_{i=1}^{n}$} \left(e_i^{(a)} \right)^4 \leq L; \ \  \hbox{$n^{-1} \sum_{i=1}^{n}$} \left( e_i^{(b)} \right)^4 \leq L.
\end{equation}
\end{condition}

\begin{condition} \label{cond:limit}
The population means $\hbox{$n^{-1} \sum_{i=1}^{n}$} \left( e_i^{(a)} \right)^2$,  $\hbox{$n^{-1} \sum_{i=1}^{n}$} \left( e_i^{(b)} \right)^2$ and $\hbox{$n^{-1} \sum_{i=1}^{n}$} e_i^{(a)} e_i^{(b)}$ converge to finite limits.
\end{condition}

\begin{thm}\label{thm:general}
Under assumptions \ref{cond:stability} to \ref{cond:limit}, suppose that the penalized estimators $\hat \bbeta^{a}_{\pen}$ and $\hat \bbeta^{b}_{\pen}$ are mean risk consistent in the following sence:
\begin{equation}
\label{eqn:risk-consistency}
\sqrt{n} \left( \bar{\bx}_{A} \right)^T \left( \hat \bbeta^{(a)}_{\pen}   - \bbeta^{(a)}   \right)  \stackrel{p}{\rightarrow} 0 , \quad  \sqrt{n} \left( \bar{\bx}_{B} \right)^T \left( \hat \bbeta^{(b)}_{\pen}   - \bbeta^{(b)}   \right)  \stackrel{p}{\rightarrow} 0.
\end{equation}
Then the penalized regression adjusted ACE estimator $\ACEhat_{\pen}$ is asymptotically unbiased and asymptotically normal, that is,
\begin{eqnarray} \label{ACE-conv}
\sqrt{n}  \left( \ACEhat_{\pen}  - \tau \right) \stackrel{d}{\rightarrow} \mathcal{N}\left( 0, \ATEsig^2 \right)
\end{eqnarray}
where
\begin{equation} \label{sig-def-theorem}
\ATEsig^2 = \lim_{n\rightarrow \infty}\left[\frac{1}{p_A} \sigma^2_{e^{(a)}} + \frac{1}{1-p_A}\sigma^2_{e^{(b)}} - \sigma^2_{e^{(a)}-e^{(b)}}\right].
\end{equation}
Moreover, the asymptotic variance, $\ATEsig^2$, is no greater than that of the difference-in-means estimator and the difference is $(1/(p_A(1-p_A))) \Delta$ where
\begin{equation}
\Delta = - \lim_{n \rightarrow \infty} \sigma^2_{X \bbeta_E} \leq 0, \quad \bbeta_E = (1-p_A)\bbeta^{(a)} + p_A \bbeta^{(b)}.
\end{equation}

\end{thm}

Theorem \ref{thm:general} follows from an intermediate results (with certain modification) in the proof of the work \cite{bloniarz2015lasso}. We give an
explicit theorem here to emphasize that this result holds for a very general class of penalized regression adjusted ACE estimator. In the following section, we will apply this theorem to study the theoretical properties of several penalized regression adjustment -- the Ridge, Elastic Net and Adaptive Lasso.

Since we cannot observe $a_i$ and $b_i$ simultaneously, the variance of their difference, $\sigma^2_{e^{(a)}-e^{(b)}}$, is not identifiable. Thus, the asymptotic variance $\ATEsig^2 $ cannot be estimated consistently. However, we can provide a Neyman-type conservative variance estimate. Let $\hat \sigma^2_{e^{(a)}}$ and $\hat \sigma^2_{e^{(b)}}$ be consistent estimators of $\sigma^2_{e^{(a)}}$ and $\sigma^2_{e^{(b)}}$ respectively (for example, using the residual sum of squares), then $\ATEsig^2 $ can be conservatively estimated by
\[
\frac{n}{n_A} \hat \sigma^2_{e^{(a)}} + \frac{n}{n_B} \hat \sigma^2_{e^{(b)}} .
\]

\subsection{Ridge adjusted ACE estimator}
In this section, we apply the general result of Theorem~\ref{thm:general} to the Ridge penalty which is defined by
\[
p_a( \beta_j ) =   \dfrac{1}{2} \lambda_{a,2}  \beta_j^2 ; \quad p_b(\beta_j ) =   \dfrac{1}{2} \lambda_{b,2} \beta_j^2.
\]
The resulting adjusted vectors are denoted by $\hat \bbeta^{(a)}_{\Ridge}$, $\hat \bbeta^{(b)}_{\Ridge}$ respectively, and the Ridge adjusted ACE estimator is 
\[ \hat{\tau}_{\Ridge}  = \left[ \bar{a}_A - \left( \bar{\bx}_A  \right)^T\hat\bbeta^{(a)}_{\Ridge} \right]
- \left[ \bar{b}_B - \left( \bar{\bx}_B  \right)^T\hat\bbeta^{(b)}_{\Ridge} \right].\]

We derive the mean risk consistency of the Ridge adjusted vector under the following assumptions. Define $\delta_n$ to be the maximum covariance between the error terms and the covariates.
\begin{equation}\label{def:delta}
\delta_n = \max_{z=a,b} \left\{ \max_j  \left|\frac{1}{n} \sum_{i=1}^{n}  x_{ij}  e^{(z)}_i  \right| \right\}.
\end{equation}

\begin{condition} \label{cond:scaling-ridge}
Suppose that
\begin{equation}
\label{cond:delta_n-ridge}
\delta_n = o\left(  \frac{ 1 }{ p \sqrt{ \log p } }  \right) ,
\end{equation}
\begin{equation}
\label{cond:s-scaling-ridge}
(p \log p)/{\sqrt n} = o(1).
\end{equation}
\end{condition}
\begin{remark}
The scaling assumption \eqref{cond:s-scaling-ridge} allows the number of covariates $p$ to converge to infinity as $n \rightarrow \infty$, but at a rate much slower than $\sqrt{n}$. We conjecture that this assumption can be weaken to at least $p/\sqrt{n} \rightarrow \infty$ by using more sophisticated technique. In practice, the Ridge can be applied even when $p$ is comparable to or larger than $n$. We leave the theoretical analysis of Ridge adjustment for higher dimension (for example $p/n \rightarrow c \in (0,1)$) to future research.
\end{remark}

\begin{condition}\label{cond:smallest-eigen}
For the Ridge Gram matrices 
\[ \Sigma^{(z)}_{\Ridge} = \Sigma + \lambda_{z, 2} I, \quad z= a,b, \] 
where $I$ is a $p \times p$ identity matrix and $\Sigma = X^TX/n$ is the Gram matrix, suppose that their smallest eigenvalues are bounded away from $0$, that is, there exists a constant $\Lambda_{\min} > 0$ such that
\begin{equation*}
\min\{ \lambda_{\min}\left(\Sigma^{(a)}_{\Ridge}  \right),  \lambda_{\min}\left(\Sigma^{(b)}_{\Ridge}  \right)\}    \geq \Lambda_{\min}.
\end{equation*}
\end{condition}

\begin{condition}\label{cond:tuning-ridge}
Assume that the regularization parameters $ \lambda_{a,2}$ and $ \lambda_{b,2} $ of the Ridge satisfy
\begin{equation}
\label{cond:lambda-a-2-ridge}
\min\{ \lambda_{a,2} || \bbeta^{(a)}  ||_1 ,   \lambda_{b,2} || \bbeta^{(b)}  ||_1  \} = O\left( p \sqrt{\frac{\log p}{n} } \right).
\end{equation}
\end{condition}

\begin{thm}\label{thm:ridge}
Under assumptions \ref{cond:stability} to \ref{cond:tuning-ridge}, the Ridge adjusted vectors $\hat \bbeta^{(a)}_{Ridge}$ and $\hat \bbeta^{(b)}_{\Ridge}$ satisfy
\begin{equation*}
\sqrt{n} \left( \bar{\bx}_{A} \right)^T \left( \hat \bbeta^{(a)}_{\Ridge}   - \bbeta^{(a)}   \right)  \stackrel{p}{\rightarrow} 0 , \quad  \sqrt{n} \left( \bar{\bx}_{B} \right)^T \left( \hat \bbeta^{(b)}_{\Ridge}   - \bbeta^{(b)}   \right)  \stackrel{p}{\rightarrow} 0.
\end{equation*}
Then, $\ACEhat_{\Ridge}$ is asymptotically normal, that is,
\begin{eqnarray} \label{ACE-conv}
\sqrt{n}  \left( \ACEhat_{\Ridge}  - \tau \right) \stackrel{d}{\rightarrow} \mathcal{N}\left( 0, \sigma^2_{\pen}  \right).
\end{eqnarray}
\end{thm}

\begin{remark}
Setting $\lambda_{a,2}=\lambda_{b,2}=0$, Theorem~\ref{thm:ridge} becomes the asymptotic normality of the Ordinary Least Squares (OLS) adjusted ACE estimator. It extends the result of \cite{lin2013} from fixed $p$ to diverging $p$ setting ($p \log p / \sqrt{n} \rightarrow 0$).
\end{remark}

\begin{remark}
Similar to linear regression model, the Ridge adjustment is better than the OLS adjustment when there exists strong multi-collinearity, that is, the smallest eigenvalue of the Gram matrix $\Sigma $ is close to 0. In this case, we can chose $ \lambda_{a,2}$ and $ \lambda_{b,2} $ such that assumption~\ref{cond:smallest-eigen} still holds, but on the other hand, they cannot be too large due to assumption~\ref{cond:tuning-ridge}. In practice, the tuning parameters can be chosen by cross validation.
\end{remark}

As mentioned in the previous section, the asymptotic variance cannot be estimated consistently since $a_i$ and $b_i$ are not observed simultaneously. However, we can give a Neyman-type conservative estimate of the variance based on the residual sum of squares. Let
\begin{equation}
\label{var_estim_a}
 \hat \sigma^2_{e^{(z)},\Ridge} = \frac{1}{n_A - 1 } \sum_{i \in A} \left( z_i - \bar{z}_A - (\bx_i - \bar{\bx}_A )^T \hat \bbeta^{(z)}_{\Ridge} \right)^2, \quad z = a, b. \nonumber
\end{equation}

Define the estimator of the asymptotic variance of $\sqrt{n}(\hat \tau_{\Ridge}  - \tau)$ as follows:
\begin{equation} \label{sigma-lasso}
\hat \sigma^2_\textnormal{Ridge} = \frac{n}{n_A}  \hat \sigma^2_{e^{(a)},\Ridge} + \frac{n}{n_B} \hat \sigma^2_{e^{(b)},\Ridge}.\nonumber
\end{equation}

\begin{condition}\label{cond:largest}
The largest eigenvalue of the Gram matrix $\Sigma = X^TX/n$ is bounded away from $\infty$, that is, there exists a constant $\Lambda_{\max} < \infty$ such that
\begin{equation*}
\lambda_{\max}\left(\Sigma\right)\leq \Lambda_{\max}.
\end{equation*}
\end{condition}

\begin{thm}
\label{thm:variance-ridge}
Under assumptions \ref{cond:stability} to \ref{cond:largest}, $\hat \sigma^2_{\Ridge}$ converges in probability to
\begin{equation*}
\frac{1}{p_A} \mathop {\lim}\limits_{n \rightarrow \infty } \sigma^2_{e^{(a)}} + \frac{1}{1-p_A} \mathop {\lim}\limits_{n \rightarrow \infty } \sigma^2_{e^{(b)}},
\end{equation*}
which is greater than or equal to the asymptotic variance of $\sqrt{n}( \ACEhat_{\Ridge}  - \tau )$. The difference is $\mathop {\lim}\limits_{n \rightarrow \infty } \sigma^2_{e^{(a)}-e^{(b)}} $.
\end{thm}

\begin{remark}
We may construct better confidence interval for the ACE if we appropriately adjust the degrees of freedom of the residual sum of squares, but it is difficult to define and consistently estimate the degrees of freedom for Ridge estimators, especially in high dimensional scenario when $p > n$. We leave this direction to future research.
\end{remark}

\subsection{Elastic Net adjusted ACE estimator}
In this section, we apply Theorem~\ref{thm:general} to the Elastic Net, and provide conditions under which the Elastic Net adjusted ACE estimator is asymptotically normal and is asymptotically more efficient than the unadjusted estimator. 

According to \cite{Zou2005}, the naive Elastic Net estimator is defined by setting the penalties
\[
p_a( \beta_j ) =  \lambda_{a,1}  | \beta_j | +  \dfrac{1}{2} \lambda_{a,2}  \beta_j^2 ; \quad p_b(\beta_j ) =  \lambda_{b,1} |\beta_j| +  \dfrac{1}{2} \lambda_{b,2} \beta_j^2,
\]
and we denote the resulting adjusted vectors and the adjusted ACE estimator as $\hat \bbeta^{(a)}_{\naiveEN}$, $\hat \bbeta^{(b)}_{\naiveEN}$, and $\hat{\tau}_{\naiveEN}$ respectively. The Elastic Net estimator is a rescaled version of the naive Elastic Net estimator which can have better prediction performance due to the paper \cite{Zou2005}. Specifically, let
\[
\hat \bbeta^{(a)}_{\EN} = \left( 1 + \lambda_{a,2} \right) \hat \bbeta^{(a)}_{\naiveEN}; \quad  \hat \bbeta^{(b)}_{\EN} = \left( 1 + \lambda_{b,2} \right) \hat \bbeta^{(b)}_{\naiveEN},
\]
the Elastic Net adjusted ACE estimator is defined as
\begin{equation}
\begin{split}
\label{def-ate-EN}
\hat{\tau}_{\EN} = &
\left[ \bar{a}_A - \left( \bar{\bx}_A - \bar{\bx} \right)^T\hat\bbeta^{(a)}_{\EN} \right]
- \left[ \bar{b}_B - \left( \bar{\bx}_B - \bar{\bx} \right)^T\hat\bbeta^{(b)}_{\EN} \right].  \nonumber
\end{split} \nonumber
\end{equation}

We first show the mean risk consistency of the Elastic Net adjusted estimators, that is, show that they satisfy \eqref{eqn:risk-consistency}. These results extend the $l_1$ consistence of the Elastic Net from sparse linear regression model to the Neyman-Rubin causal model. Given the regularization parameter $\lambda_{a,1}, \lambda_{b,1}$ and $\bbeta^{(a)}$ and $\bbeta^{(b)}$, the sparsity measures for treatment and control groups, $s^{(a)}_{\lambda_{a,1}}$ and $s^{(b)}_{\lambda_{b,1}}$ are defined as
\begin{equation}\label{def:s}
s^{(a)}_{\lambda_{a,1}} = \sum_{j=1}^p \min\left\{ \frac{ |\beta_j^{(a)}| }{\lambda_{a,1}}, 1 \right\} \ \text{and} \  s^{(b)}_{\lambda_{b,1}} = \sum_{j=1}^p \min\left\{ \frac{ |\beta_j^{(b)}| }{\lambda_{b,1}}, 1 \right\}.
\end{equation}
We note that $s^{(a)}_{\lambda_{a,1}}$ and $s^{(b)}_{\lambda_{b,1}}$ are allowed to grow as $n$ increases\footnote{In this definition of sparsity, we allow that there exit some small but nonzero components of $\bbeta^{(a)}$ and $\bbeta^{(b)}$ which decay at a fast enough rate.}. For notational simplicity, we do note index the parameters with $n$.

The following conditions will guarantee that the Elastic Net consistently estimates the adjustment vectors $\bbeta^{(a)}, \bbeta^{(b)}$ at a fast enough rate. These conditions are similar to but weaker than those that guarantee the $l_1$ consistence of the Lasso.

\begin{condition} \label{cond:scaling}
Let $s = \max \left\{ s^{(a)}_{\lambda_{a,1}}, s^{(b)}_{\lambda_{b,1}} \right\}$, suppose that
\begin{equation}
\label{cond:delta_n}
\delta_n = o\left(  \frac{ 1 }{ s \sqrt{\log p} }  \right),
\end{equation}
\begin{equation}
\label{cond:s-scaling}
(s \log p)/{\sqrt n} = o(1).
\end{equation}
\end{condition}

\begin{condition}\label{last-cond} Define the shrinked covariance matrices of the covariates as
\begin{equation}
\Sigma^{(z)}_{\EN} = (1+\lambda_{z,2})^{-1}\left( \Sigma + \lambda_{z,2} I\right), \quad z= a,b,
\end{equation}
where $I$ is a $p \times p$ identity matrix. Suppose there exist constants $C >0$ and $\xi >1$ not depending on $n$, such that
\begin{equation}\label{cond:cone-inv}
\|\bh_S \|_1 \leq Cs  \min\left\{ \| \Sigma^{(a)}_{\EN} \bh \|_\infty ,  \| \Sigma^{(b)}_{\EN} \bh \|_\infty\right\}, \ \forall \ \bh \in \mathcal{C},
\end{equation}
with $\mathcal{C}=\{\bh: \|\bh_{S^c}\|_1 \leq \xi \|\bh_{S}\|_1\}$, and
\begin{equation} \label{def:S}
S = \{ j: |\beta^{(a)}_j| > \lambda_{a,1} \ \textnormal{or} \  |\beta^{(b)}_j| > \lambda_{b,1} \}.
\end{equation}
\end{condition}


\begin{remark}
This assumption is the major difference between the Lasso \cite{bloniarz2015lasso} and the Elastic Net. Setting $\lambda_{a,2}= \lambda_{b,2} = 0$ leads it back to the cone invertibility factor condition in \cite{bloniarz2015lasso} which basically assumes that the smallest restricted eigenvalue of the Gram matrix $\Sigma = X^TX/n$ is bounded away from 0 at a certain rate. The matrices $\Sigma^{(a)}_{\EN}$ and $\Sigma^{(b)}_{\EN}$ shrink the  $\Sigma$ towards the identity matrix and when the restricted eigenvalue of the $\Sigma$ is close to 0, the restricted eigenvalues of the shrinked matrices can still be away from 0. Thus, this condition is weaker than the cone invertibility factor condition, making the Elastic Net adjusted ACE estimator performs better  than the Lasso adjusted ACE estimator.
\end{remark}

\begin{condition}\label{last-last-cond}
Let $\varsigma = \min\big\{1/70,(3p_A)^2/70,(3-3p_A)^2/70 \big\}$. For constants $0< \eta < \frac{\xi - 1}{\xi + 1}$ and $\frac{1}{\eta}<M<\infty$, and for $z=a,b$, assume the regularization parameters of the Elastic Net belong to the sets\footnote{We denote $p_a = p_A, p_b = p_B$ for notation simplicity.}

\small{
\begin{equation}\label{cond:lambda-a}
\lambda_{z,1} \in   \left( \frac{1}{\eta}, M \right] \times \left( \frac{2(1+\varsigma)L^{1/2}}{p_z} \sqrt{ \frac{2\log p}{n} } + \delta_n + \frac{ \lambda_{z,2} } { 1 + \lambda_{z,2} } ( L^{1/2} + 1 ) || \bbeta^{(z)}  ||_1 \right),
\end{equation}
\begin{equation}
\label{cond:lambda-a-2}
\max\{ \lambda_{a,2} || \bbeta^{(a)}  ||_1 ,   \lambda_{b,2} || \bbeta^{(b)}  ||_1  \} = O \left(  \sqrt{\frac{\log p}{n} } \right).
\end{equation}
}
\end{condition}

\begin{remark}
Under assumption \ref{last-last-cond}, $\lambda_{a,2}, \lambda_{b,2}$ are restricted to take relatively small values\footnote{We cannot choose a much larger value for them because in that case the $l_2$ penalty would become dominant, and no estimates will be set to zero.}. Together with assumption \ref{last-cond}, the theoretical advantages of the Elastic Net adjustment over the Lasso adjustment may not be very prominent. 
\end{remark}

\begin{remark}
Setting $\lambda_{a,2} = \lambda_{b,2} = 0$, assumptions \ref{cond:scaling} - \ref{last-last-cond} are the same as those proposed in paper \cite{bloniarz2015lasso} for deriving the asymptotic normality of the Lasso adjusted ACE estimator.
\end{remark}

\begin{thm}\label{mean risk consistent}
Under assumptions \ref{cond:stability} to \ref{cond:limit} and \ref{cond:scaling} - \ref{last-last-cond}, the Elastic Net adjusted vectors $\hat \bbeta^{(a)}_{\EN}$ and $\hat \bbeta^{(b)}_{\EN}$ satisfy
\begin{equation*}
\sqrt{n} \left( \bar{\bx}_{A} \right)^T \left( \hat \bbeta^{(a)}_{\EN}   - \bbeta^{(a)}   \right)  \stackrel{p}{\rightarrow} 0 , \quad  \sqrt{n} \left( \bar{\bx}_{B} \right)^T \left( \hat \bbeta^{(b)}_{\EN}   - \bbeta^{(b)}   \right)  \stackrel{p}{\rightarrow} 0.
\end{equation*}
Then, $\ACEhat_{\EN}$ is asymptotically unbiased and asymptotically normal, that is,
\begin{eqnarray} \label{ACE-conv}
\sqrt{n}  \left( \ACEhat_{\EN}  - \tau \right) \stackrel{d}{\rightarrow} \mathcal{N}\left( 0, \sigma^2_{\pen}  \right).
\end{eqnarray}
\end{thm}

\begin{remark}
Bunea F. showed in his paper \cite{Bunea2008} that, in sparse linear regression model, the Elastic Net enjoys a smaller $l_1$ error bound under the same conditions as the Lasso, which is equivalent to show that the Elastic Net can obtain the same $l_1$ error bound under weaker conditions than the Lasso. This is what we have obtained in the above theorem. The main difference between our results and Bunea's lies on we operate within the Neyman-Rubin model with fixed potential outcomes for a finite population, where the treatment group is sampled without replacement.
\end{remark}

We next give a Neyman-type conservative estimate of the asymptotic variance based on the residual sum of squares. For $z=a,b$, let
\begin{equation}
\label{var_estim_a}
 \hat \sigma^2_{e^{(z)}} = \frac{1}{n_A - df^{(z)} } \sum_{i \in A} \left( z_i - \bar{z}_A - (\bx_i - \bar{\bx}_A )^T \hat \bbeta^{(z)}_{\EN} \right)^2,   \nonumber
\end{equation}
where $df^{(a)}$ and $df^{(b)}$ are degrees of freedom defined by
\[ df^{(z)} = \hat s^{(z)}+1 = || \hat \bbeta^{(z)}_\textnormal{EN} ||_0 +1,  \quad z = a,b. \]

Define the variance estimate as follows:
\begin{equation} \label{sigma-lasso}
\hat \sigma^2_\textnormal{EN} = \frac{n}{n_A}  \hat \sigma^2_{e^{(a)}} + \frac{n}{n_B} \hat \sigma^2_{e^{(b)}}.\nonumber
\end{equation}

\begin{thm}
\label{conservative_variance}
Under assumptions  \ref{cond:stability} to \ref{cond:limit} and \ref{cond:largest} - \ref{last-last-cond}, $\hat \sigma^2_\textnormal{EN}$ converges in probability to
\begin{equation*}
\frac{1}{p_A} \mathop {\lim}\limits_{n \rightarrow \infty } \sigma^2_{e^{(a)}} + \frac{1}{1-p_A} \mathop {\lim}\limits_{n \rightarrow \infty } \sigma^2_{e^{(b)}},
\end{equation*}
which is greater than or equal to the asymptotic variance of $\sqrt{n}( \ACEhat_{\EN}  - \tau )$. The difference is  $\mathop {\lim}\limits_{n \rightarrow \infty } \sigma^2_{e^{(a)}-e^{(b)}} $.
\end{thm}

\begin{remark}
If we redefine $\hat \sigma^2_{e^{(a)}}$ and $\hat \sigma^2_{e^{(b)}}$ without adjusting the degrees of freedom:
\[ (\hat \sigma^*)^2_{e^{(z)}} = \frac{1}{n_A } \sum_{i \in A} \left( z_i - \bar{z}_A - (\bx_i - \bar{\bx}_A )^T \hat \bbeta^{(z)}_\textnormal{EN} \right)^2, \quad z = a,b.  \]
And define
\[(\hat \sigma^*)^2_\textnormal{EN} = \frac{n}{n_A}  (\hat \sigma^*)^2_{e^{(a)}} + \frac{n}{n_B} (\hat \sigma^*)^2_{e^{(b)}}.\]
We can obtain the same asymptotic result of the variance estimate $(\hat \sigma^*)^2_\textnormal{EN}$ as in Theorem~\ref{conservative_variance}. However, our simulation study shows that adjusting the degree of freedom can result in better confidence interval for the ACE.

\end{remark}

\subsection{Adaptive Lasso adjusted ACE estimator}
In the section, we illustrate how to apply Theorem~\ref{thm:general} to the Adaptive Lasso adjusted ACE estimator. We outline another set of assumptions that guarantee the mean risk consistency of the Adaptive Lasso estimator. 

By \cite{Zou2006}, the Adaptive Lasso adds different amount of penalties on different components of the coefficients, taking into account the relative importance of the covariates:
\[
p_a( \beta_j ) =  \lambda_{a,1} w_{a,j}  | \beta_j |  ; \quad p_b(\beta_j ) =  \lambda_{b,1} w_{b,j} |\beta_j| ,
\]
where $w_{a,j}, w_{b,j} \geq 0$ are some weights whose values could be taken as the inverse of the absolute values of the OLS estimator in low dimension or the Lasso estimator in high dimension or any other $\sqrt{n}$-consistent initial estimator. The main advantage of the Adaptive Lasso lies on it has variable selection consistency under much weaker conditions than the Lasso. The Adaptive Lasso adjusted vectors and adjusted ACE estimator are denoted as $\hat \bbeta^{(a)}_{\ada}$, $\hat \bbeta^{(b)}_{\ada}$, and $\hat{\tau}_{\ada}$ respectively. 
\begin{equation}
\begin{split}
\label{def-ate-ada}
\hat{\tau}_{\ada} = &
\left[ \bar{a}_A - \left( \bar{\bx}_A - \bar{\bx} \right)^T\hat\bbeta^{(a)}_{\ada} \right]
- \left[ \bar{b}_B - \left( \bar{\bx}_B - \bar{\bx} \right)^T\hat\bbeta^{(b)}_{\ada} \right].  \nonumber
\end{split} \nonumber
\end{equation}

Under similar assumptions of Theorem~\ref{mean risk consistent}, we can derive the mean risk consistency of the Adaptive Lasso estimator.
\begin{condition}\label{last-cond-ada} Let $z = a, b$, define the weighted covariance matrices of the covariates as
\begin{equation}
\Sigma^{(z)}_{\ada} = n^{-1}\sum_{i=1}^n ( \bv^{(z)} \bx_i )( \bv^{(z)} \bx_i  )^T , 
\end{equation}
where $\bv^{(z)}$ is a diagonal matrix with $(1/w_{z,1}, \cdots, 1/w_{z,p})$ on the diagonal. Suppose there exist constants $C >0$ and $\xi >1$ not depending on $n$, such that
\begin{equation}\label{cond:cone-inv-ada}
\|\bh_S \|_1 \leq Cs  \min\left\{ \| \Sigma^{(a)}_{\ada} \bh \|_\infty ,  \| \Sigma^{(b)}_{\ada} \bh \|_\infty\right\}, \ \forall \ \bh \in \mathcal{C},
\end{equation}
with $\mathcal{C}=\{\bh: \|\bh_{S^c}\|_1 \leq \xi \|\bh_{S}\|_1\}$.
\end{condition}


\begin{remark}
This assumption becomes the cone invertibility factor condition proposed in \cite{bloniarz2015lasso} when the weights are equal, $w_{a,1} = \cdots = w_{a,p} = 1 = w_{b,1} = \cdots = w_{b,p}$. These two conditions are very similar to each other, resulting similar theoretical properties of the Adaptive Lasso and the Lasso in terms of estimation and prediction accuracy. In practice, if one carefully choses the weights, the Adaptive Lasso may, but not always, have smaller prediction errors.
\end{remark}

\begin{thm}\label{thm:ada}
Suppose assumptions \ref{cond:stability}, \ref{cond:moment} for the rescaled covariates $\bw^{(z)} \bx_i$, \ref{cond:limit}, \ref{cond:scaling},  \ref{last-last-cond} with $\lambda_{a,2} = \lambda_{b,2} = 0$, and \ref{last-cond-ada} hold, then the Adaptive Lasso adjusted vectors $\hat \bbeta^{(a)}_{\ada}$ and $\hat \bbeta^{(b)}_{\ada}$ satisfy
\begin{equation*}
\sqrt{n} \left( \bar{\bx}_{A} \right)^T \left( \hat \bbeta^{(a)}_{\ada}   - \bbeta^{(a)}   \right)  \stackrel{p}{\rightarrow} 0 , \quad  \sqrt{n} \left( \bar{\bx}_{B} \right)^T \left( \hat \bbeta^{(b)}_{\ada}   - \bbeta^{(b)}   \right)  \stackrel{p}{\rightarrow} 0.
\end{equation*}
Therefore, $\ACEhat_{\ada}$ is asymptotically unbiased and asymptotically normal, that is,
\begin{eqnarray} \label{ACE-conv-ada}
\sqrt{n}  \left( \ACEhat_{\ada}  - \tau \right) \stackrel{d}{\rightarrow} \mathcal{N}\left( 0, \ATEsig^2 \right)
\end{eqnarray}
\end{thm}

We can provide a similar Neyman-type conservative estimate of the asymptotic variance based on the residual sum of squares of the Adaptive Lasso estimator, $\hat \sigma^2_\textnormal{ada}$, just replacing the Elastic Net adjusted vectors by the Adaptive Lasso adjusted vectors.

\begin{thm}
\label{thm:variance-ada}
Under the assumptions of Theorem~\ref{thm:ada} and assumption \ref{cond:largest}, $\hat \sigma^2_\textnormal{ada}$ converges in probability to
\begin{equation*}
\frac{1}{p_A} \mathop {\lim}\limits_{n \rightarrow \infty } \sigma^2_{e^{(a)}} + \frac{1}{1-p_A} \mathop {\lim}\limits_{n \rightarrow \infty } \sigma^2_{e^{(b)}},
\end{equation*}
which is greater than or equal to the asymptotic variance of $\sqrt{n}( \ACEhat_{\ada}  - \tau )$.
\end{thm}

\section{Simulation}

In this section we conduct simulation studies to evaluate the finite sample performance of the unadjusted difference-in-means estimator and the penalized regression adjusted ACE estimators including the Lasso, Ridge, naive Elastic Net, Elastic Net and Adaptive Lasso. We use the R package ``elasticnet" to compute the Elastic Net solution path and use the package ``glmnet" to compute the Adaptive Lasso and the Ridge estimators. The tuning parameters are selected by 10-fold Cross Validation (CV).

Similar to the simulation setups in \cite{bloniarz2015lasso}, the potential outcomes $a_i$ and $b_i$ are generated from the following nonlinear model: for $i=1,...,n$,
\[ a_i = \sum_{j=1}^{s} x_{ij} \beta_{j}^{(a)} + \exp{ \left( \sum_{j=1}^s x_{ij} \beta_{j}^{(a)} \times 0.15 \right) } + \epsilon^{(a)}_i, \]
\[ b_i = \sum_{j=1}^{s} x_{ij} \beta_{j}^{(b)} + \exp{ \left( \sum_{j=1}^s x_{ij} \beta_{j}^{(b)} \times 0.15 \right) } + \epsilon^{(b)}_i, \]
where $\epsilon^{(a)}_i$ and $\epsilon^{(b)}_i$ are independent error terms generated from normal distribution $N(0,\sigma^2)$. We set $n=200$, $p=50$ and $500$. For $p=50$, we include the OLS adjusted estimator in the comparison. The covariate vectors $\bx_i, i=1,\cdots,n$ are generated independently from a multivariate normal distribution $\mathcal{N}(0,\Sigma)$. We consider four examples for generating $\bx$ and $\bbeta$.
\begin{enumerate}
\item We set $s = 10$, $\sigma = 3$, the correlation between $x_i$ and $x_j$ equals $0.85^{ | i - j |}$, $i, \ j = 1,\cdots,p$, and
\[ \beta_{j}^{(a)} = 0.5, \ j=1,2,\cdots,10; \quad    \beta_{j}^{(b)} = 0.25, \ j=1,2,\cdots,10. \]
\item The same as Example 1 except that the first 10 nonzero elements of $\beta_{j}^{(a)}$ and $\beta_{j}^{(b)}$ are generated form uniform distribution $U(0,1)$.
\item The same as Example 1 expect that the correlation between  $x_i$ and $x_j$ equals $0.75$ for all $i,\ j = 1,\cdots,p, \ i \neq j$.
\item This example is similar to example 4 in the Elastic Net paper \cite{Zou2005}, but we are interested in estimating the ACE instead of the prediction performance of the Elastic Net in linear regression model. In this example, we set $\sigma =2$, $s = 15$, and
\begin{eqnarray}
\beta_j^{(a)} = \left \{
\begin{array}{ll}
0.5 & j = 1,\cdots, 5, \\
0.75 & j = 6, \cdots, 10, \\
1 & j = 11, \cdots, 15,
\end{array}
\right. \quad \quad \beta_j^{(b)} = \beta_j^{(a)} - 0.25, \quad j = 1, \cdots, 15. \nonumber
\end{eqnarray}
The covariates are generated as follows:
\begin{eqnarray}
&& \bx_i = W_1 + \epsilon_i^{\bx}, \quad W_1 \sim N(0,1), \quad i = 1, \cdots, 5, \nonumber \\
&& \bx_i = W_2 + \epsilon_i^{\bx}, \quad W_2 \sim N(0,1), \quad i = 6, \cdots, 10, \nonumber \\
&& \bx_i = W_3 + \epsilon_i^{\bx}, \quad W_3 \sim N(0,1), \quad i = 11, \cdots, 15, \nonumber \\
&& \bx_i \sim N(0,1), \quad \textnormal{$\bx_i$ independent and identically distributed}, \ i = 16, \cdots, p, \nonumber
\end{eqnarray}
where $\epsilon_i^{\bx}$ are independent and identically distributed $N(0,0.01), i = 1,\cdots,15$. In this example, there are three equally important groups with five members in each group.
\end{enumerate}

In all the examples, the values of $\bx_i$, $\bbeta^{(a)}$, $\bbeta^{(b)}$, $ \epsilon^{(a)}_i$, $ \epsilon^{(b)}_i$, $a_i$ and $b_i$ are generated once and then kept fixed. After they are generated, a completely randomized experiment is simulated $1000$ times, assigning $n_A= 80, 100, 120$ subjects to treatment group $A$ and the remainder to control group $B$. There are $24$ different combinations in total. We will use the abbreviations ``EN" for Elastic Net, ``naiveEN" for naive Elastic Net, and ``Ada" for Adaptive Lasso in the following Tables and Figures. 

The results of different ACE estimators are shown in the boxplots \ref{fig:boxplot1} - \ref{fig:boxplot4} with their variances presented on top of each method. We can see that the OLS adjusted ACE estimator dramatically decreases the variance of the unadjusted ACE estimator when $p < n$ and the penalized regression adjusted estimators (Lasso, Ridge, Elastic Net and Adpative Lasso) further reduce the variance for both $p<n$ and $p>n$ settings. Compared with the Lasso adjustment, the Elastic Net adjusted ACE estimator is at least $5\%$ better in most cases\footnote{This is not always true and there do exist simulation setups where the Elastic Net adjusted ACE estimator is worse than the Lasso adjusted ACE estimator.}, but the improvement is not as significant as the Lasso adjusted ACE estimator over the unadjusted one. The Ridge adjusted ACE estimator performs slightly better than the Lasso when $p =50$, while for $p=500$, it performs much worse unless the covariates are highly correlated and the magnitude of each coefficient is almost the same (see Example 3). Moreover, there seems no improvement of the Adaptive Lasso adjustment over the Lasso adjustment.

\begin{figure*}
\centerline{\includegraphics[width=\textwidth]{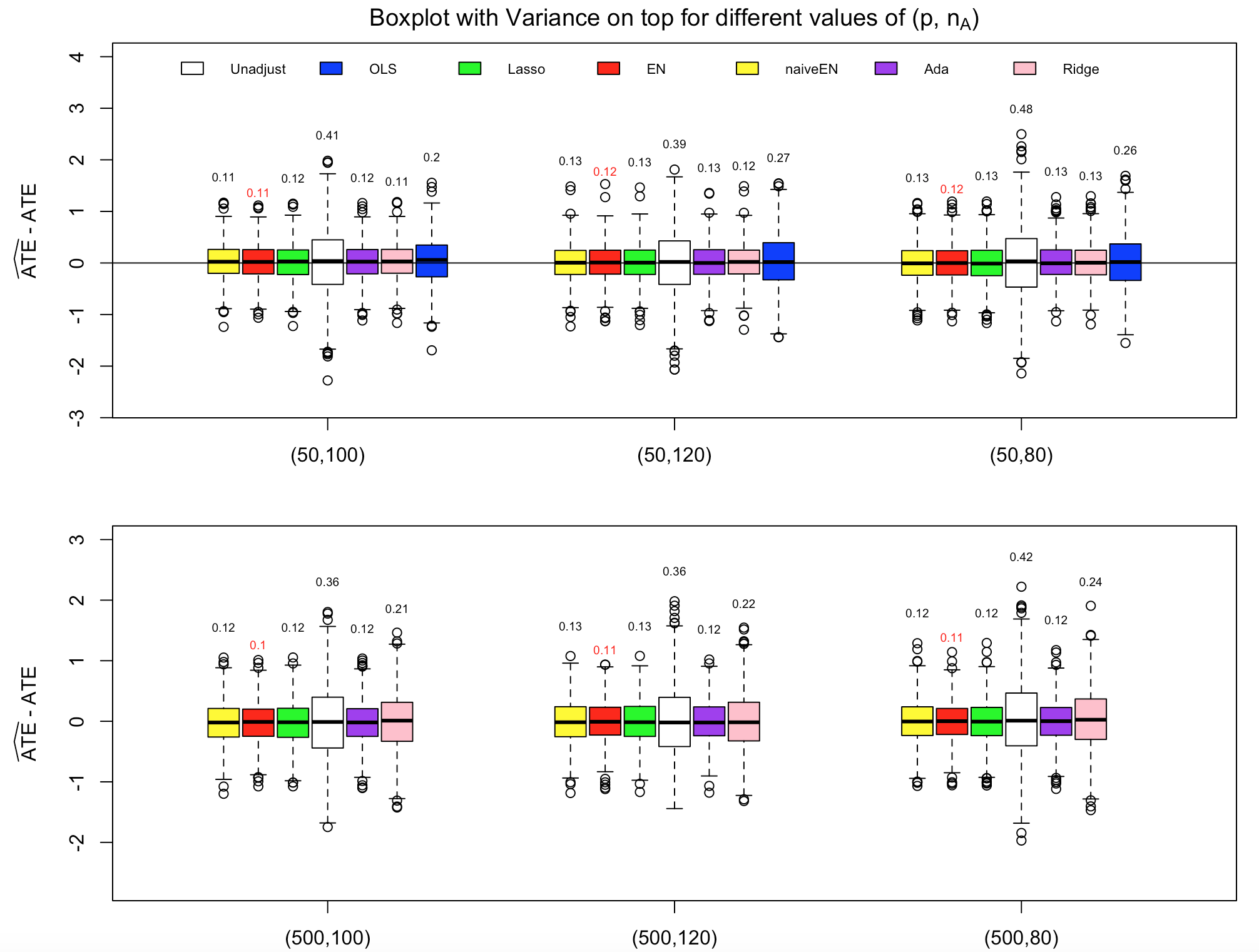}}
\vspace*{0.05in}
\caption{Boxplot of the unadjusted (white), OLS adjusted (blue; only computed when $p=50$), Lasso adjusted (green), Elastic Net adjusted (red), naive Elastic Net adjusted (yellow), Adaptive Lasso adjusted (purple), and Ridge adjusted (pink) ACE estimators with their variances presented on top of each box for Example 1.}\label{fig:boxplot1}
\end{figure*}

\begin{figure*}
\centerline{\includegraphics[width=\textwidth]{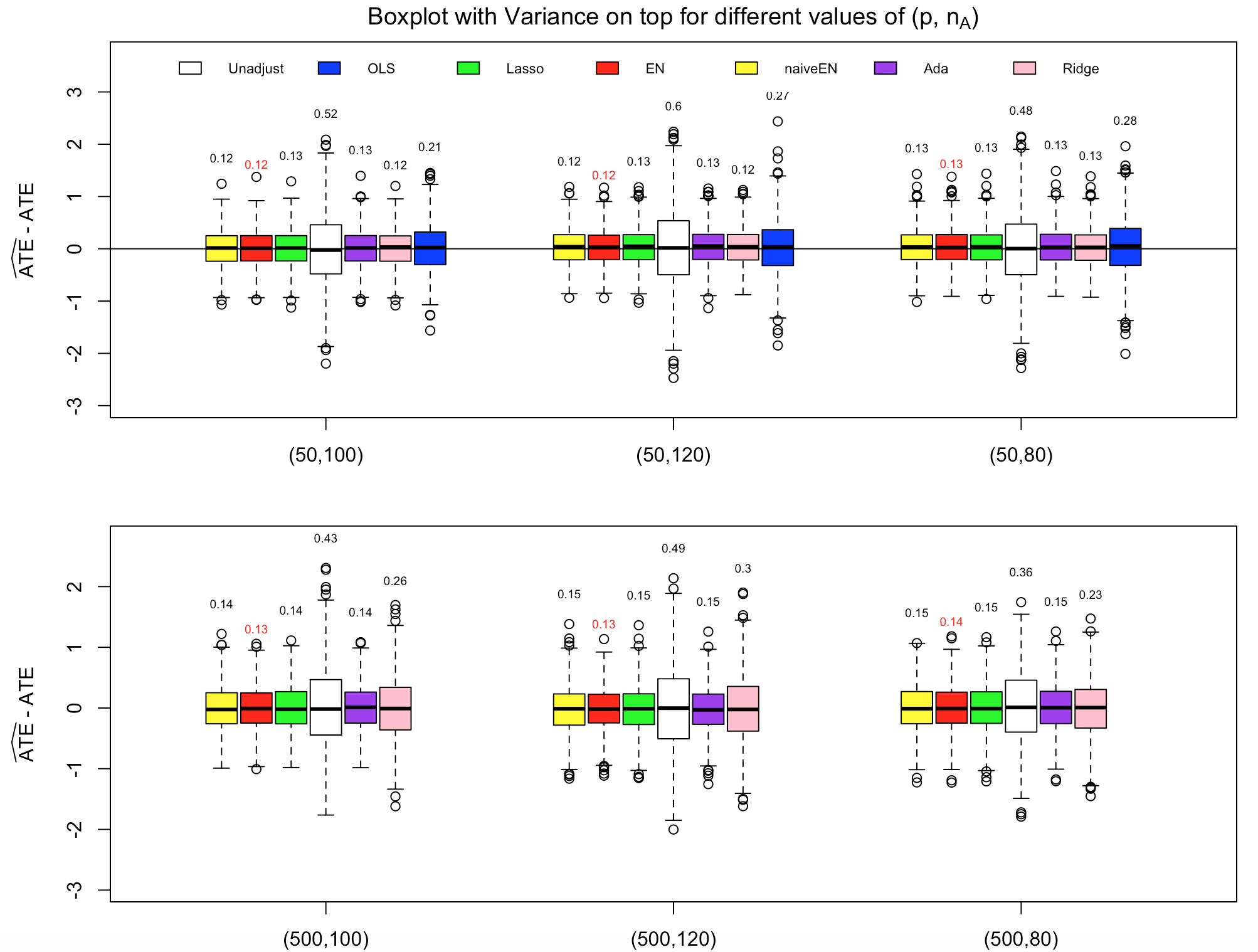}}
\vspace*{0.05in}
\caption{Boxplot of the unadjusted (white), OLS adjusted (blue; only computed when $p=50$), Lasso adjusted (green), Elastic Net adjusted (red), naive Elastic Net adjusted (yellow), Adaptive Lasso adjusted (purple), and Ridge adjusted (pink) ACE estimators with their variances presented on top of each box for Example 2.}\label{fig:boxplot2}
\end{figure*}

\begin{figure*}
\centerline{\includegraphics[width=\textwidth]{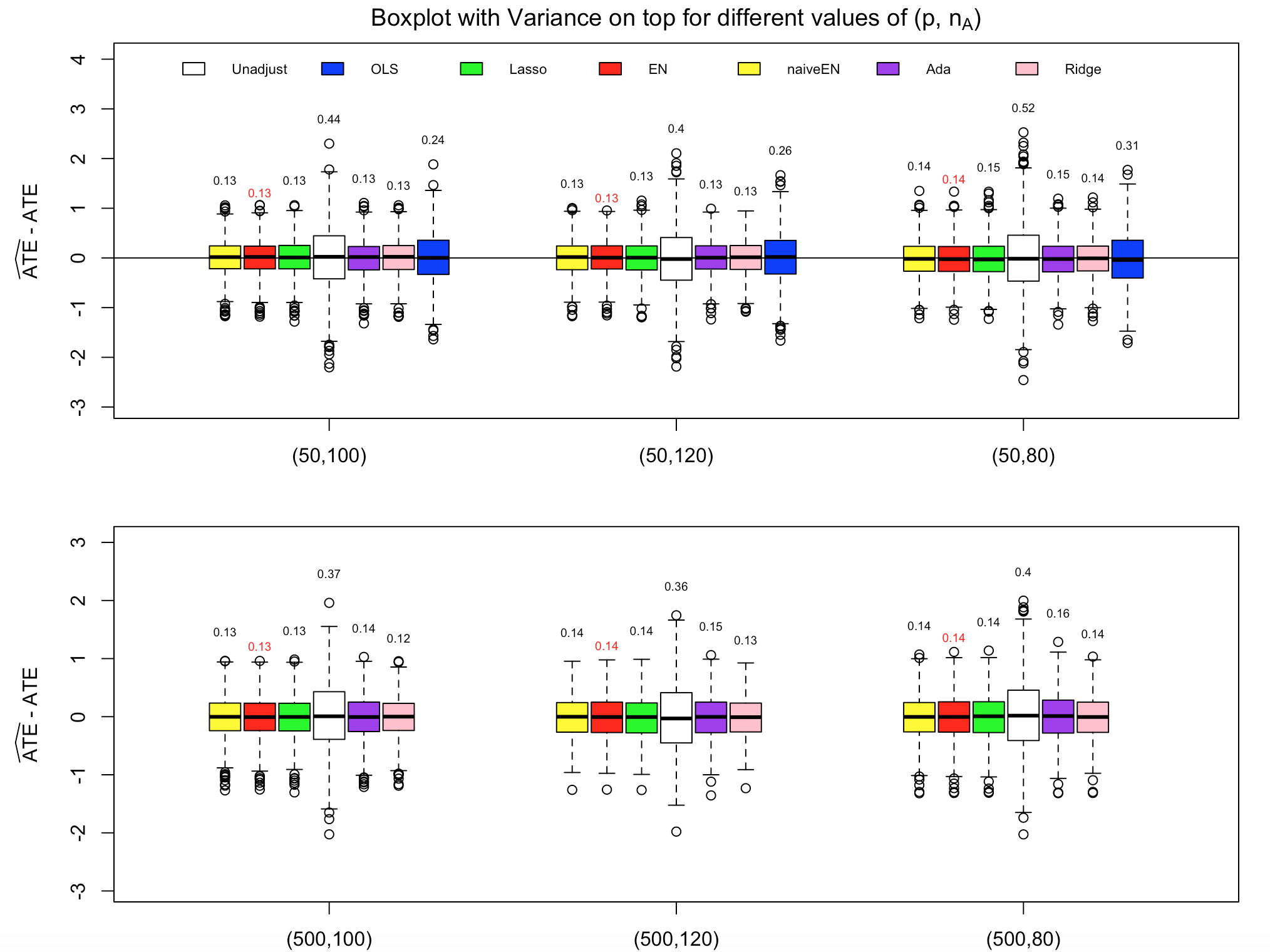}}
\vspace*{0.05in}
\caption{Boxplot of the unadjusted (white), OLS adjusted (blue; only computed when $p=50$), Lasso adjusted (green), Elastic Net adjusted (red), naive Elastic Net adjusted (yellow), Adaptive Lasso adjusted (purple), and Ridge adjusted (pink) ACE estimators with their variances presented on top of each box for Example 3.}\label{fig:boxplot3}
\end{figure*}

\begin{figure*}
\centerline{\includegraphics[width=\textwidth]{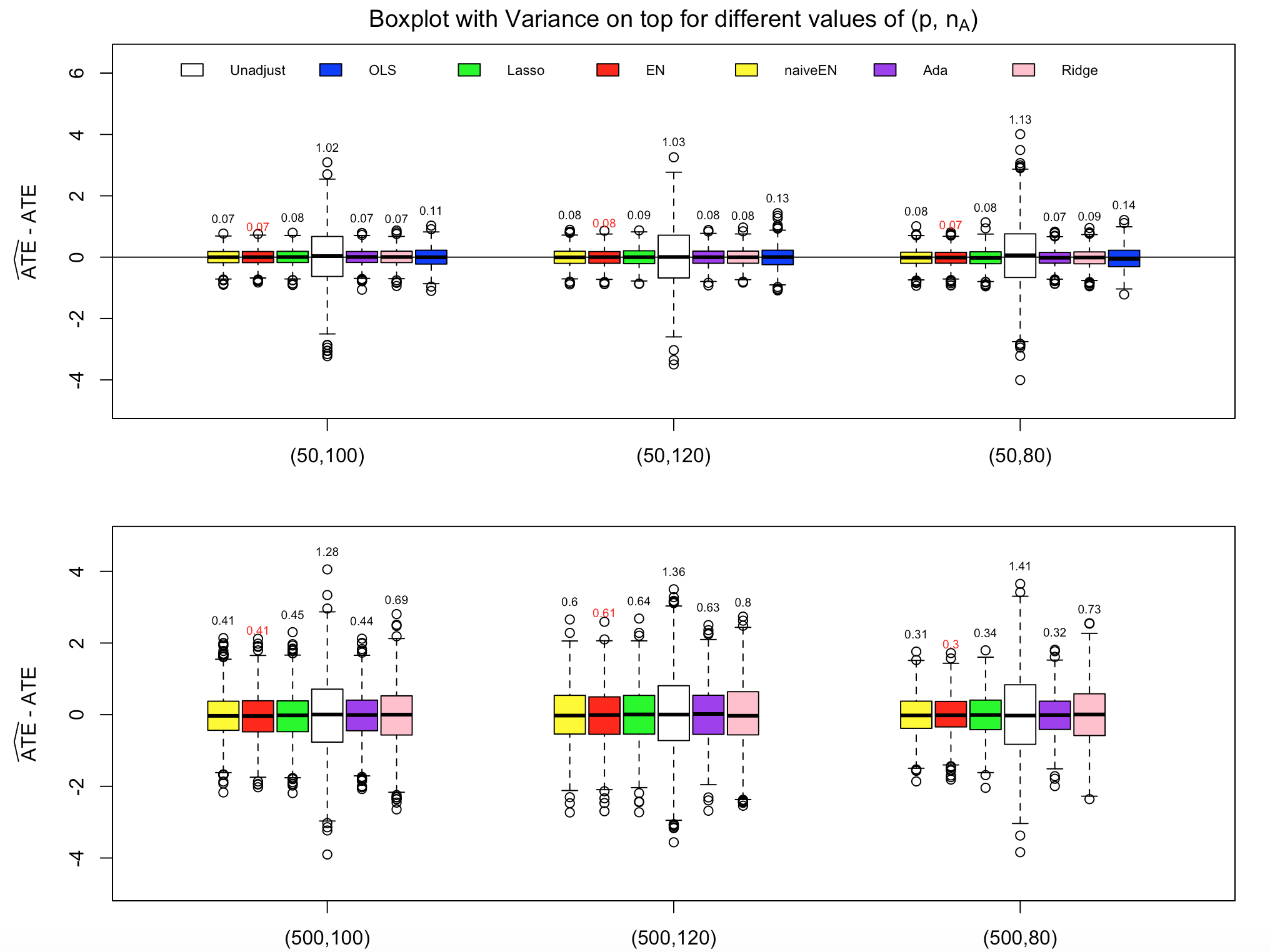}}
\vspace*{0.05in}
\caption{Boxplot of the unadjusted (white), OLS adjusted (blue; only computed when $p=50$), Lasso adjusted (green), Elastic Net adjusted (red), naive Elastic Net adjusted (yellow), Adaptive Lasso adjusted (purple), and Ridge adjusted (pink) ACE estimators with their variances presented on top of each box for Example 4.}\label{fig:boxplot4}
\end{figure*}

To see the comparison results in more details, we present the square of the bias (Bias$^2$), the variance (Var), the mean square error (MSE), the empirical coverage probability (Coverage) and the mean interval length (Length) for $95\%$ confidence interval of different estimators in Tables~\ref{tab:mse1} - \ref{tab:mse4} in the Appendix \ref{app:tables}. We find that the Bias$^2$ of each method is substantially smaller (more than 100 times) than the variance, implying that the finite sample bias is not a problem for regression and penalized regression adjusted ACE estimators. The OLS adjusted estimator reduces the MSE of the unadjusted estimator by $30\% - 80\%$ when $p=50$ and the Lasso adjusted estimator further reduces the MSE which are $60\% - 90\%$ (for $p=50$) and $60\% - 75\%$ (for $p=500$) smaller than that of the unadjusted estimator. The Elastic Net adjusted estimator performs at least as well as the Lasso adjusted estimator and it outperforms the latter for most cases, especially when there are group covariates and the correlations between the covariates within each group is high. On average, the efficiency gain is around $5\% - 14\%$ for $p = 50$ and $0.7\% - 13\%$ for $p = 500$. The naive Elastic Net adjusted estimator performs worse (less than $5\%$) than the Lasso adjusted estimator, similar to the findings in the original Elastic Net paper \cite{Zou2005}. Compared with the Lasso, the Ridge adjustment is better (less than $5\%$) for $p=50$ and is worse ($25\% - 110\%$) for $p=500$ except Example 3. In our simulation setups, the performance of the Adaptive Lasso adjustment is comparable to the Lasso adjustment and no one dominates the other.

In terms of the finite sample performance of the Neyman-type conservative variance estimates, we find that, in the first three examples, the coverage probabilities of all the methods except the OLS are much larger than the pre-assigned significance level ($95\%$), showing that these confidence intervals are very conservative. However, in Example 4 when there exist group variables and $p=500$, the variance estimators for the Lasso, Ridge and Elastic Net adjustments are not large enough to obtain the desired significance level, although their MSEs are still much better than the unadjusted estimator. In this case, the Adaptive Lasso works the best. We require to develop alternative variance estimation methods to improve the coverage probabilities in future work.

The Ridge adjusted estimator has the shortest confidence interval in all cases except $p=500$ in Example 4. But the comparison is not very fair to other methods because the Ridge does not adjust for the degrees of freedom while other methods do. Moreover, when $p=500$, its coverage probability is not satisfactory. 


We also note that the mean interval lengths of the Elastic Net adjusted estimator is slightly longer than that of the Lasso adjusted estimator. The reason is the Elastic Net seems to lose fitting accuracy a little bit in order to obtain better prediction, therefore, the residual sum of squares of the Elastic Net is a little bit larger than that of the Lasso, resulting in a little bit larger estimation of the variance. We may use the conservative variance estimate of the Lasso adjusted estimator to construct confidence interval for the Elastic Net adjusted estimator since the asymptotic variance of the Elastic Net adjusted estimator is no more than that of the Lasso adjusted estimator. Otherwise, we need to explore better statistical method for adjusting the degree of freedom of the Elastic Net estimator or use the bootstrap technique.

Overall, the Elastic Net adjusted ACE estimator performs the best in our simulations, followed by the Lasso, and then the Adaptive Lasso and finally the Ridge which is more suitable to the setting where the correlations between covariates are large and each coefficient has similar size.

\section{Conclusion}

Investigators often use regression adjustments to analyze randomized experiments in order to improve the estimation accuracy of various causal effect estimates. In this paper, we study the theoretical properties of penalized regression adjusted average causal effect (ACE) estimators when the number of baseline covariates in the experiments are large. Our analysis is presented within the Neyman-Rubin model with fixed potential outcomes for a finite population, where the treatment group is sampled without replacement.

We first derive a unified theorem on asymptotic normality for these estimators and outline conditions under which they are asymptotically more efficient than the simple difference-in-means estimator. 

We then apply the unified theorem to three penalized regression adjusted ACE estimators: the Ridge, Elastic Net and Adaptive Lasso. Specifically, we derive the asymptotic normality of the Ridge adjusted ACE estimator when $p \log p / \sqrt{n}  \rightarrow \infty$ and other appropriate assumptions hold. We show that the Elastic Net adjusted ACE estimator can be as efficient as the Lasso adjusted ACE estimator but under weaker conditions. We find that the theoretical advantage of the Adaptive Lasso adjusted ACE estimator over the Lasso adjusted ACE estimator is not significant. In theory, the Adaptive Lasso is often far better than the initial Lasso as regards variable selection, but its prediction and estimation performance are similar to the initial Lasso. Thus, the Adaptive Lasso is potentially useful for estimating heterogeneous causal effect such as the effect of treatment on subgroup of the units in the experiments, since it requires to correctly identify the set of relevant covariates\footnote{As pointed out by \cite{bloniarz2015lasso}, ``the covariates selected by the Lasso is unstable, and this may cause problems when interpreting them as evidence of heterogeneous treatment effect".} for this purpose. We leave it as future work to explore the properties of Adaptive Lasso adjusted heterogeneous causal effect. 

Lastly, we provide Neyman-type conservative estimates of the asymptotic variances of the penalized regression adjusted ACE estimators. Our variance estimators are based on the residual sum of squares with appropriately adjusted degrees of freedom, which result in asymptotically conservative confidence intervals for the ACE. They work well in most simulation cases, but there is still much space for improvement when the covariates are made up of several important groups and they are highly correlated within each group.

\section*{Acknowledgements}
Dr. Hanzhong Liu's research is supported by the National Natural Science Foundation of China (Grant No. 11701316). Dr. Yuehan Yang's research is partially supported by the National Natural Science Foundation of China (Grant No. 11671059).

\appendix

\section{Proof of the results}\label{app:proofs}

\subsection{Proof of Theorem~\ref{thm:general}}

\begin{proof}
Without loss of generality, assume that
\begin{eqnarray}\label{zero-means}
\bar{a}=0, \ \bar{b}=0.
\end{eqnarray}
Otherwise, we consider $a_i - \bar{a} $ and $b_i - \bar{b}$.
Recall the decompositions of the potential outcomes:
\begin{equation}\label{eqn:a}
a_i = \bar{a} + ( \bx_i  )^T \bbeta^{(a)} + \err^{(a)}_i = \bx_i^T \bbeta^{(a)} + \err^{(a)}_i,
\end{equation}
\begin{equation}
b_i = \bar{b} + ( \bx_i  )^T \bbeta^{(b)} + \err^{(b)}_i = \bx_i^T \bbeta^{(b)} + \err^{(b)}_i.
\end{equation}
If we define $\bh^{(a)}_{\pen} = \hat{\bbeta}^{(a)}_{\pen} - \bbeta^{(a)}$, $\bh^{(b)}_{\pen} = \hat{\bbeta}^{(b)}_{\pen} - \bbeta^{(b)}$, by substitution, we have
\begin{equation*}
\begin{split}
&\sqrt{n} ( \ACEhat_{\pen} - \tau )  =  \underbrace{ \sqrt{n} \left[ \bar{\err}_A^{(a)} - \bar{\err}_B^{(b)} \right] }_{\hypertarget{ATE-lhs}{}*} -  \underbrace{ \left[ \sqrt{n}  \left( \bar{\bx}_A \right)^T{\bh^{(a)}_{\pen}} -  \sqrt{n}  \left( \bar{\bx}_B \right)^T{\bh^{(b)}_{\pen}} \right ]}_{\hypertarget{ATE-rhs}{}**}.
\end{split}
\end{equation*}
Asymptotic normality of \hyperlink{ATE-lhs}{$(*)$} follows from Theorem 1 in \cite{freedman2008regression_a} with $a$ and $b$ replaced by $e^{(a)}$ and $e^{(b)}$ respectively. That is,
\begin{eqnarray} \label{ACE-conv}
\sqrt{n} \left[ \bar{\err}_A^{(a)} - \bar{\err}_B^{(b)} \right] \stackrel{d}{\rightarrow} \mathcal{N}\left( 0, \ATEsig^2 \right)
\end{eqnarray}
where
\begin{equation} \label{sig-def}
\ATEsig^2 = \lim_{n\rightarrow \infty}\left[\frac{1}{p_A} \sigma^2_{e^{(a)}} + \frac{1}{1-p_A}\sigma^2_{e^{(b)}} - \sigma^2_{e^{(a)}-e^{(b)}}\right].\end{equation}
Under the assumption \ref{eqn:risk-consistency}, \hyperlink{ATE-rhs}{$(**)$} converges to $0$ in probability. Thus, the asymptotic normality of Theorem \ref{thm:general} holds.

Next, we compare the asymptotic variance, $\ATEsig^2$, with that of the unadjusted difference-in-means estimator. Let $\bbeta^{(a)} = \bbeta^{(b)} = 0$, we can obtain the asymptotic normality of the difference-in-means estimator whose asymptotic variance is 
\begin{equation} \label{sig-def}
\ATEunadj^2 = \lim_{n\rightarrow \infty}\left[\frac{1}{p_A} \sigma^2_{a} + \frac{1}{1-p_A}\sigma^2_{b} - \sigma^2_{a-b}\right].\end{equation}
Since $\bbeta^{(a)}$ and $\bbeta^{(b)}$ are projection coefficients, we have the approximation errors $e^{(a)}$ and $e^{(b)}$ are orthogonal to $X \bbeta^{(a)}$ and $X \bbeta^{(b)}$. Therefore,
\[  \sigma_{a}^2  = \sigma^2_{X\bbeta^{(a)}} + \sigma^2_{e^{(a)}}, \quad  \sigma_{b}^2  = \sigma^2_{X\bbeta^{(b)}} + \sigma^2_{e^{(b)}}, \]
\[  \sigma_{a-b}^2  = \sigma^2_{X\bbeta^{(a)} - X\bbeta^{(b)}} + \sigma^2_{e^{(a)} - e^{(b)}} . \]
Hence, the difference of the asymptotic variances is the limit of 
\small{
\begin{eqnarray}
& &  \sigma^2_{X\bbeta^{(a)} - X\bbeta^{(b)}} - \sigma^2_{X\bbeta^{(a)}}/p_A - \sigma^2_{X\bbeta^{(b)}}/(1-p_A) \nonumber \\
& = &  \sum_{i=1}^{n}  \left[ \frac{ p_A ( 1 - p_A) \left( \bx_i^T \bbeta^{(a)} - \bx_i^T \bbeta^{(b)} \right)^2 - (1 - p_A) \left(  \bx_i^T \bbeta^{(a)} \right)^2 - p_A \left( \bx_i^T \bbeta^{(b)} \right)^2 }{(n -1 )p_A ( 1 - p_A )} \right] \nonumber \\
& = &  \sum_{i=1}^{n} \frac{ - \left( \bx_i^T \left( (1 - p_A ) \bbeta^{(a)} + p_A \bbeta^{(b)} \right) \right)^2 }{ (n -1 )p_A ( 1 - p_A ) } = \sum_{i=1}^{n} \frac{ - \left( \bx_i^T \bbeta_E \right)^2 }{ (n -1 )p_A ( 1 - p_A ) } = \frac{ - \sigma^2_{X\bbeta_E} }{p_A ( 1 - p_A )}. \nonumber
\end{eqnarray}
}

\end{proof}

\subsection{Some Useful Lemmas}

Before the proof of the main theorems, we first state several useful lemmas on upper bounds of $\left\Vert \bar{\bx}_A \right\Vert_\infty$, $\left\Vert (\bar{\bx}_A) (\bar{e}_A) \right\Vert_\infty$ and so on. The proofs of these lemmas could be found in \cite{bloniarz2015lasso}.

The first lemma is the Massart concentration inequality for sampling without replacement, which is the basis for various theoretical analysis of penalized estimators under finite population asymptotic framework. 

\begin{lem}[Bloniarz, Liu, Zhang, Sekhon and Yu (2016)]
\label{thm:concentration}
Let $\{z_i, i=1,...,n\}$
be a finite population of real numbers. Let $A \subset \{i,\ldots,n\}$ be a subset of deterministic size $|A|=n_A$ that is selected randomly without replacement. Define $p_A = n_A/n, \  \sigma^2 = \hbox{$n^{-1} \sum_{i=1}^{n}$} (z_i-\bar{z})^2$. Then, for any $t > 0$,
\begin{eqnarray}\label{bernstein}
P\left(\bar{z}_A   - \bar{z} \geq t \right) \leq \exp \left\{ - \frac{p_A n_A t^2}{(1+ \varsigma )^2 \sigma^2} \right\},
\end{eqnarray}
with $\varsigma = \min\big\{1/70,(3p_A)^2/70,(3-3p_A)^2/70 \big\}$.
\end{lem}


The following lemmas are directly obtained from the Massart concentration inequality (Lemma~\ref{thm:concentration}), we state the results here without proof.
\begin{lem}[Bloniarz, Liu, Zhang, Sekhon and Yu (2016)]
\label{lem:subsamplemean}
Let $\{z_i, i=1,...,n\}$ be a finite population of real numbers. Let $A \subset \{i,\ldots,n\}$ be a subset of deterministic size $ |A| = n_A$ that is selected randomly without replacement. Suppose that the population mean of the $z_i$ has a finite limit and that there exist constants $\epsilon>0$ and $L<\infty$ such that
\begin{equation}
\label{con:eps}
\frac{1}{n} \sum_{i=1}^{n} |z_i|^{1+\epsilon} \leq L.
\end{equation}
If $\frac{n_A}{n} \rightarrow p_A \in (0,1)$, then
\begin{equation}
\bar{z}_A  \stackrel{p}{\rightarrow}  \mathop {\lim}\limits_{n \rightarrow \infty } \bar{z}.
\end{equation}
\end{lem}

Let 
\[ \hat \Sigma_A =  \frac{1}{n_A} \sum_{i\in A} \left( \bx_i  - \bar{\bx}_A  \right) \left( \bx_i  - \bar{\bx}_A  \right)^T, \quad  \Sigma = \frac{1}{n}\sum_{i=1}^{n} \bx_i \bx_i^T. \]

\begin{lem}[Bloniarz, Liu, Zhang, Sekhon and Yu (2016)]
\label{lem:concentrationofterms}
Under assumptions \ref{cond:stability}, \ref{cond:moment} and \ref{cond:scaling}, and as $n \rightarrow \infty$, we have
\small{
		\begin{equation}
		\label{lem:xterm-ada}
		P\left( \left\Vert \bar{\bx}_A  \right\Vert_\infty >  \frac{(1+\varsigma )L^{1/4}}{n_A} \sqrt{\frac{2\log p}{n}} \right) \rightarrow 0,
		\end{equation}
		
		\begin{equation}
		P\left( \left\Vert  \bar{e}_A^{(a)}  \right\Vert_\infty >  \frac{(1+\varsigma )L^{1/4}}{n_A} \sqrt{\frac{2\log p}{n}} \right) \rightarrow 0,
		\end{equation}
		
\begin{equation} \label{xeterm}
P\left( \shortnorm{ \frac{1}{n_A} \sum_{i \in A} \left( \bx_i - \bar{\bx}_A \right) \left(  e_i^{(a)}  - \bar{e}_A^{(a)} \right) }_\infty > \frac{2(1+\varsigma)L^{1/2}}{p_A} \sqrt{ \frac{2\log p}{n} } + \delta_n \right)  \rightarrow 0,
\end{equation}

\begin{eqnarray} \label{cov-mat-part1}
P\left(  \shortnorm{ \hat \Sigma_A   - \Sigma  }_\infty \geq  \frac{2(1+\varsigma)L^{1/2}}{p_A} \sqrt{ \frac{\log p}{n} } \right) \rightarrow 0.  \nonumber \\
\end{eqnarray}
}
\end{lem}

\subsection{Proof of Theorems}

It is enough to show the mean risk consistency of the Ridge, Elastic Net and Adaptive Lasso adjusted vectors. In this part we will focus on $\hat \bbeta^{(a)}$ for treatment group $A$, as the same analysis can be applied to $\hat \bbeta^{(b)}$ for control group $B$. The proof follows similar arguments of the Lemma 3 in \cite{bloniarz2015lasso} with necessary modifications dealing with different penalties.  For convenience, we will drop the superscript ``$(a)$" on quantities such as $\hat \bbeta^{(a)}$, $\bbeta^{(a)}$ and $\err^{(a)}$.

\subsubsection{Proof of Theorem~\ref{thm:ridge}}

\begin{proof}
By the definition of Ridge estimator,
\begin{eqnarray}
\hat \bbeta_{\Ridge} & =  & \argmin_{\bbeta} \frac{1}{2n_A} \sum_{i \in A} \left( a_i - \bar{a}_A - \left( \bx_i - \bar{\bx}_A \right)^T \bbeta \right)^2 + \frac{1}{2} \lambda_{a,2} \sum_{j=1}^{p} \beta_j^2   \nonumber \\
& = & \left( \hat \Sigma_A + \lambda_{a,2} I \right)^{-1}  \left[  \frac{1}{n_A} \sum_{i \in A} \left( \bx_i - \bar{\bx}_A \right) \left( a_i - \bar{a}_A \right) \right] \nonumber \\
& = & \left( \hat \Sigma_A + \lambda_{a,2} I \right)^{-1}  \left[  \hat \Sigma_A \bbeta + \frac{1}{n_A} \sum_{i \in A} \left( \bx_i - \bar{\bx}_A \right) \left( \err_i - \bar{\err}_A  \right) \right] \nonumber \\
& = & \bbeta +  \left( \hat \Sigma_A + \lambda_{a,2} I \right)^{-1} \left[  \frac{1}{n_A} \sum_{i \in A} \left( \bx_i - \bar{\bx}_A \right) \left( \err_i - \bar{\err}_A  \right) \right] - \lambda_{a,2} \bbeta, \nonumber
\end{eqnarray}
where $I$ is a $p \times p$ identity matrix and the last but second equality is due to the decomposition of the potential outcomes
\[a_i = \bar{a} + ( \bx_i  )^T \bbeta + \err_i,  \quad  \bar{a}_A  = \bar{a} +  ( \bar{\bx}_A  )^T \bbeta + \bar{\err}_A. \]
Therefore,
\begin{eqnarray}
&& \shortnorm{\hat \bbeta_{\Ridge} - \bbeta}_1 \nonumber \\
&\leq& \shortnorm{\lambda_{a,2} \bbeta}_1 + \norm{ \left( \hat \Sigma_A + \lambda_{a,2} I \right)^{-1} \left[  \frac{1}{n_A} \sum_{i \in A} \left( \bx_i - \bar{\bx}_A \right) \left( \err_i - \bar{\err}_A  \right) \right] }_1 \nonumber \\
& \leq & \shortnorm{\lambda_{a,2} \bbeta}_1 + \sqrt{p}  \norm{ \left( \hat \Sigma_A + \lambda_{a,2} I \right)^{-1} \left[  \frac{1}{n_A} \sum_{i \in A} \left( \bx_i - \bar{\bx}_A \right) \left( \err_i - \bar{\err}_A  \right) \right] }_2 \nonumber \\
& \leq & \shortnorm{\lambda_{a,2} \bbeta}_1 + \sqrt{p} \lambda_{\max} \left( \left( \hat \Sigma_A + \lambda_{a,2} I \right)^{-1} \right) \norm{   \frac{1}{n_A} \sum_{i \in A} \left( \bx_i - \bar{\bx}_A \right) \left( \err_i - \bar{\err}_A  \right)  }_2 \nonumber \\
& \leq &  \shortnorm{\lambda_{a,2} \bbeta}_1 + \sqrt{p} O_p\left(2/\Lambda_{\min} \right) \sqrt{p}  \shortnorm{   \frac{1}{n_A} \sum_{i \in A} \left( \bx_i - \bar{\bx}_A \right) \left( \err_i - \bar{\err}_A  \right)  }_\infty, \nonumber \\
& =  & \shortnorm{\lambda_{a,2} \bbeta}_1 + O_p\left(p/\Lambda_{\min} \right) \shortnorm{   \frac{1}{n_A} \sum_{i \in A} \left( \bx_i - \bar{\bx}_A \right) \left( \err_i - \bar{\err}_A  \right)  }_\infty
\end{eqnarray}
where the last inequality is due to the following Lemma:
\begin{lem}
\label{lem:eigen-ridge}
Under the assumptions of Theorem~\ref{thm:ridge}, 
\[ P \left(  \lambda_{\max} \left( \left( \hat \Sigma_A + \lambda_{a,2} I \right)^{-1} \right) \leq 2/\Lambda_{\min} \right)  \rightarrow 1 .\]
\end{lem}
By Lemma~\ref{lem:concentrationofterms},
\begin{equation*} 
P\left( \shortnorm{ \frac{1}{n_A} \sum_{i \in A} \left( \bx_i - \bar{\bx}_A \right) \left(  e_i  - \bar{e}_A \right) }_\infty > \frac{2(1+\varsigma)L^{1/2}}{p_A} \sqrt{ \frac{2\log p}{n} } + \delta_n \right)  \rightarrow 0,
\end{equation*}
together with the scaling assumption~\ref{cond:scaling-ridge}, we have
\[   \shortnorm{ \frac{1}{n_A} \sum_{i \in A} \left( \bx_i - \bar{\bx}_A \right) \left(  e_i^{(a)}  - \bar{e}_A^{(a)} \right) }_\infty  = O_p \left(   \sqrt{ \frac{\log p}{n} } + \delta_n  \right).  \]
By assumption~\ref{cond:tuning-ridge} and the scaling assumption~\ref{cond:scaling-ridge},  
\[  \shortnorm{\lambda_{a,2} \bbeta}_1  = O\left(  p \sqrt{ \frac{\log p}{n} } \right) =  o \left(  \frac{1}{\sqrt{\log p}} \right) , \]
thus,
\[   \shortnorm{\hat \bbeta_{\Ridge} - \bbeta}_1 =  O_p\left(  p \sqrt{ \frac{\log p}{n} }  + p \delta_n \right) = o_p \left(  \frac{1}{\sqrt{\log p}} \right). \]
Therefore, together with Lemma~\ref{lem:concentrationofterms} and the scaling assumption~\eqref{cond:s-scaling-ridge}, the Ridge adjusted vector $\hat \bbeta_{\Ridge}$ is mean risk consistent, that is,
\begin{eqnarray*}
\left | \sqrt{n} \left( \bar{\bx}_A \right)^T  \left(  \hat \bbeta_{\Ridge}  - \bbeta \right) \right| & \leq &   \sqrt{n}  \shortnorm{ \bar{\bx}_A }_\infty  \shortnorm{  \hat \bbeta_{\Ridge}  - \bbeta}_1   \nonumber \\
& = &   \sqrt{n}  O_p \left( \sqrt{ \frac{\log p}{n} }  \right)  o_p \left(  \frac{1}{\sqrt{\log p}} \right) = o_p(1). 
\end{eqnarray*}
\end{proof}

\subsubsection{Proof of Theorem~\ref{thm:variance-ridge}}

\begin{proof}
The proof is almost the same as proof of Theorem~\ref{conservative_variance} (replacing $s$ by $p$ and replacing the Elastic Net adjusted vector by Ridge adjusted vector), so we omit it here.
\end{proof}

\subsubsection{Proof of Theorem \ref{mean risk consistent}}

\begin{proof}
We start with the KKT condition, which characterizes the solution to the Elastic Net. Recall the definition of the naive Elastic Net estimator $\hat \bbeta_{\naiveEN}$:
\[ \hat \bbeta_{\naiveEN} =  \argmin_{\bbeta} \frac{1}{2n_A} \sum_{i \in A} \left( a_i - \bar{a}_A - (\bx_i -  \bar{\bx}_A )^T \bbeta \right)^2 + \lambda_{a,1} \shortnorm{\bbeta}_1+ \frac{1}{2}\lambda_{a,2} \shortnorm{\bbeta}_2^2 \]
The KKT condition for $\hat \bbeta_{\naiveEN} $ is
\begin{equation}
\label{KKT}
 \frac{1}{n_A} \sum_{i\in A} (\bx_i -  \bar{\bx}_A) \left( a_i -  \bar{a}_A  - (\bx_i -  \bar{\bx}_A)^T \hat \bbeta_{\naiveEN} \right) - \lambda_{a,2} \hat \bbeta_{\naiveEN} = \lambda_{a,1} \mathbf{\kappa},
\end{equation}
where $\kappa$ is the subgradient of $||\bbeta||_1$ taking value at $\bbeta = \hat \bbeta_{\naiveEN}$, i.e.,
\begin{equation}\label{eqn:KKT}
\begin{split}
\mathbf{\kappa} \in \partial || \bbeta ||_1 \left |_{ \bbeta = \hat \bbeta_{\naiveEN} } \right.  \quad \textnormal{with} \quad
\left\{
\begin{aligned}
\kappa_j &\in [-1,1] \textnormal{ for } j \st \hat\beta_{\naiveEN,j} = 0 \\
\kappa_j &= \textnormal{sign}\left(\hat\beta_{\naiveEN,j}\right) \textnormal{ otherwise.}
\end{aligned}
\right.
\end{split}
\end{equation}
Since
\[a_i = \bar{a} + ( \bx_i  )^T \bbeta + \err_i, \]
we have
\begin{eqnarray}
\label{KKT11}
&& \left( \hat \Sigma_A +  \lambda_{a,2}I\right) \left( \bbeta  - \hat \bbeta_{\naiveEN}  \right) +  \frac{1}{n_A} \sum_{i\in A} (\bx_i - \bar{\bx}_A)( e_i - \bar{e}_A )  - \lambda_{a,2}\bbeta   = \lambda_{a,1} \mathbf{\kappa}. \nonumber \\
\end{eqnarray}
By definition,
\[\hat \bbeta_{\EN} = (1+ \lambda_{a,2}) \hat \bbeta_{\naiveEN},\]
then
\begin{eqnarray}
\label{KKT1}
&& (1+\lambda_{a,2})^{-1}\left( \hat \Sigma_A + \lambda_{a,2}I\right) \left( ( 1+\lambda_{a,2} ) \bbeta  - \hat \bbeta_{\EN}  \right)  \nonumber \\
& & +  \frac{1}{n_A} \sum_{i\in A} (\bx_i - \bar{\bx}_A)( e_i - \bar{e}_A )  - \lambda_{a,2}\bbeta  = \lambda_{a,1} \mathbf{\kappa}.
\end{eqnarray}
Denoting
\[   
\hat \Sigma^{(a)}_{\EN} = (1+\lambda_{a,2})^{-1}\left( \hat \Sigma_A + \lambda_{a,2}I\right),
\]
we have
\begin{eqnarray}
\label{KKT1}
&& \hat \Sigma^{(a)}_{\EN}   \left( \bbeta - \hat \bbeta_{\EN}  \right) + \frac{1}{n_A} \sum_{i\in A} (\bx_i - \bar{\bx}_A)( e_i - \bar{e}_A )  - \lambda_{a,2} \left( I -  \hat \Sigma^{(a)}_{\EN}  \right) \bbeta  = \lambda_{a,1} \mathbf{\kappa}, \nonumber 
\end{eqnarray}
where $I$ is an identity matrix. Multiplying both sides of the above equation by $ - \bh^T = \left(  \bbeta  - \hat \bbeta_{\EN}  \right)^T$, we have
\begin{eqnarray}
\label{basic-inequality}
 & & \bh^T \hat \Sigma^{(a)}_{\EN}\bh
  -  \bh^T \left[ \frac{1}{n_A} \sum_{i\in A} (\bx_i - \bar{\bx}_A)( e_i - \bar{e}_A ) \right]  +  \bh^T \left[ \lambda_{a,2} \left( I -  \hat \Sigma^{(a)}_{\EN}  \right) \bbeta \right]
 \nonumber \\
 & = &  \lambda_{a,1} \left(\bbeta -  \hat \bbeta_{\EN}  \right)^T \mathbf{\kappa}     \leq  \lambda_{a,1}  \left( \shortnorm{\bbeta}_1 - \shortnorm{\hat\bbeta_{\EN}}_1\right),\nonumber
\end{eqnarray}
where the last inequality holds because
 \[ \left( \bbeta \right)^T \mathbf{\kappa} \leq || \bbeta ||_1 || \mathbf{\kappa} ||_\infty \leq || \bbeta ||_1 \ \  \textnormal{and} \ \  \hat \bbeta_{\EN}^T \mathbf{\kappa} = || \hat \bbeta_{\EN} ||_1. \]
Applying H\"older's inequality, we have
\begin{eqnarray*}
\bh^T \hat \Sigma^{(a)}_{\EN}\bh
&\leq  &   \lambda_{a,1}  \left( \shortnorm{\bbeta}_1 - \shortnorm{\hat\bbeta_{\EN}}_1\right)  \nonumber \\
& + &  \shortnorm{\bh}_1  \left( \shortnorm{ \frac{1}{n_A} \sum_{i\in A} (\bx_i - \bar{\bx}_A)( e_i - \bar{e}_A ) }_\infty + || \lambda_{a,2} \left( I -  \hat \Sigma^{(a)}_{\EN}  \right) \bbeta ||_\infty    \right).
\end{eqnarray*}
Define
\[ 
\mathcal{L} = \left\{ \shortnorm{\frac{1}{n_A} \sum_{i \in A} ( \bx_i - \bar{\bx}_A ) ( e_i - \bar{e}_A ) }_\infty  + || \lambda_{a,2} \left( I -  \hat \Sigma^{(a)}_{\EN}  \right) \bbeta ||_\infty \leq \eta \lambda_{a,1}  \right\}
\].
\begin{lem}
\label{lem:eventL}
Under the assumptions of Theorem~\ref{mean risk consistent}, $ P( \mathcal{L} ) \rightarrow 1$.
\end{lem}
We can continue our proof conditional on the event $\mathcal{L}$ and it holds that
\begin{align} \label{on-event-L}
\bh^T \hat \Sigma_{a,\EN}\bh  &\leq
\lambda_{a,1} \left( \shortnorm{\bbeta}_1 - \shortnorm{\hat\bbeta_{\EN}}_1\right) + \eta \lambda_{a,1} \shortnorm{\bh}_1.
\end{align}
By substituting the definition of $\bh$, we have
\begin{align*}
\shortnorm{\bbeta}_1 - \shortnorm{\hat\bbeta_{\EN}}_1 \leq \shortnorm{\bh_S}_1 - \shortnorm{\bh_{S^c}}_1 + 2\shortnorm{\bbeta_{S^c}}_1.
\end{align*}
Therefore,
\begin{eqnarray*}
0 \leq \bh^T \hat \Sigma_{a,\EN}\bh & \leq & \lambda_{a,1} \left( \shortnorm{\bh_S}_1 - \shortnorm{\bh_{S^c}}_1 + 2\shortnorm{\bbeta_{S^c}}_1 + \eta \shortnorm{\bh}_1 \right) \\
&\leq & \lambda_{a,1} \left[ ( \eta - 1 ) \shortnorm{\bh_{S^c}}_1 + (1 + \eta)\shortnorm{\bh_S}_1 + 2 \shortnorm{\bbeta_{S^c}}_1 \right].
\end{eqnarray*}
We obtain
\begin{equation}\label{eqn:hscbound}
\begin{split}
 ( 1 - \eta ) \shortnorm{\bh_{S^c}}_1  \leq (1 + \eta)\shortnorm{\bh_S}_1 + 2 \shortnorm{\bbeta_{S^c}}_1 \leq (1 + \eta)\shortnorm{\bh_S}_1 + 2 s \lambda_{a,1}.
\end{split}
\end{equation}
where the last inequality holds because of the definition of $s$ in \eqref{def:s} and $S$ in \eqref{def:S}. The remaining proof is almost the same as in \cite{bloniarz2015lasso}. We will present it here for completeness. Consider the following two cases:

(I) If $(1+\eta)\shortnorm{\bh_S}_1 + 2s\lambda_{a,1} \geq (1-\eta)\xi\shortnorm{\bh_S}_1$ then by \eqref{eqn:hscbound},
\begin{equation*}
\begin{split}
\shortnorm{\bh}_1  & = \shortnorm{\bh_S}_1 + \shortnorm{\bh_{S^c}}_1  \leq \left( \frac{1+\eta}{1-\eta} + 1\right)\shortnorm{\bh_S}_1 + \frac{2s\lambda_{a,1}}{1-\eta} \\
& \leq \frac{2s\lambda_{a,1}}{1-\eta}\left( \frac{2}{(1-\eta)\xi - (1+\eta)} + 1 \right).
\end{split}
\end{equation*}
By the definition of $\lambda_{a,1}$ and the scaling assumptions \eqref{cond:delta_n}, \eqref{cond:s-scaling}, we have that $s \lambda_{a,1} = o \left(  \frac{1}{\sqrt{\log p}} \right) $. Thus, 
\[ \shortnorm{\bh}_1  = o_p \left(  \frac{1}{\sqrt{\log p}} \right).  \]

(II) If $(1+\eta)\shortnorm{\bh_S}_1 + 2s\lambda_{a,1} < (1-\eta)\xi\shortnorm{\bh_S}_1$ then by  \eqref{eqn:hscbound} we have
$
\shortnorm{\bh_{S^c}}_1 \leq \xi \shortnorm{\bh_S}_1
$.
Applying assumption~\ref{last-cond} on the design matrix \eqref{cond:cone-inv},
\begin{equation}\label{eqn:hbound}
\begin{split}
\shortnorm{\bh}_1 & = \shortnorm{\bh_S}_1 + \shortnorm{\bh_{S^c}}_1  \leq (1+\xi)\shortnorm{\bh_S}_1 \leq (1+\xi) C s \shortnorm{  \Sigma^{(a)}_{\EN} \bh}_\infty
\end{split}
\end{equation}

Recall that, by KKT condition and the definition of $ - \bh^T = \left(  \bbeta  - \hat \bbeta_{\EN}  \right)^T$, we have obtained 
\begin{eqnarray}
\label{KKTnew}
&& - \hat \Sigma^{(a)}_{\EN}   \bh + \frac{1}{n_A} \sum_{i\in A} (\bx_i - \bar{\bx}_A)( e_i - \bar{e}_A )  - \lambda_{a,2} \left( I -  \hat \Sigma_{\EN}  \right) \bbeta  = \lambda_{a,1} \mathbf{\kappa}, \nonumber 
\end{eqnarray}
This time we will take the $l_\infty$-norm, yielding
\begin{eqnarray}\label{eqn:KKT-infty}
\shortnorm{   \hat \Sigma^{(a)}_{\EN}   \bh }_\infty & \leq & \lambda_{a,1} + \shortnorm{ \frac{1}{n_A} \sum_{i\in A} (\bx_i - \bar{\bx}_A)( e_i - \bar{e}_A ) }_\infty + \shortnorm{  \lambda_{a,2} \left( I -  \hat \Sigma^{(a)}_{\EN}  \right) \bbeta }_\infty  \nonumber \\
& \leq & (1+\eta) \lambda_a, 
\end{eqnarray}
where the latter inequality holds on the set $\mathcal{L}$. The final step is to control the deviation of the subsampled covariance matrix from the population covariance matrix, so that we can apply \eqref{eqn:hbound}. We define another event with constant $C_1 = \frac{2(1+\varsigma)L^{1/2}}{p_A} $,
\begin{equation*}
\begin{split}
\mathcal{M} = & \left\{
\shortnorm{  \hat \Sigma_{\EN}^{(a)}  - \Sigma_{\EN}^{(a)} }_\infty   \leq C_1  \sqrt{ \frac{\log p}{n} }
\right\}
\end{split}
\end{equation*}

\begin{lem}\label{lem:eventM}
Under the assumptions of Theroem~\ref{mean risk consistent}, we have $P(\mathcal{M}) \rightarrow 1$.
\end{lem}

We will prove Lemma \ref{lem:eventM} later. Continuing our inequalities, on the event $\mathcal{L} \cap \mathcal{M}$,
\begin{equation*}
\begin{split}
 s \shortnorm{ \Sigma_{\EN}^{(a)} \bh}_\infty  & \leq  C_1  s \sqrt{ \frac{\log p}{n} } \shortnorm{\bh}_1 + s \shortnorm{  \hat \Sigma^{(a)}_{\EN}   \bh  }_\infty \leq  o(1) \shortnorm{\bh}_1 + s(1+\eta)\lambda_{a,1},
\end{split}
\end{equation*}
where we have applied the scaling assumption \eqref{cond:s-scaling} and \eqref{eqn:KKT-infty} in the second line. Hence, by \eqref{eqn:hbound},
\begin{equation*}
\shortnorm{\bh}_1 \leq (1+\xi)C\left[ o(1)\shortnorm{\bh}_1 + s(1+\eta)\lambda_{a,1} \right].
\end{equation*}
Again, applying the scaling assumptions \eqref{cond:delta_n} and \eqref{cond:s-scaling}, we get 
\[ \shortnorm{\bh}_1 = o_p \left(  \frac{1}{\sqrt{\log p}} \right).  \]

In either case of (I) or (II), we have shown that 
\[ \shortnorm{\bh}_1 = \shortnorm{  \hat \bbeta_{\EN}  - \bbeta }_1 =  o_p \left(  \frac{1}{\sqrt{\log p}} \right).  \]
Therefore, together with Lemma~\ref{lem:concentrationofterms}, the Elastic Net adjusted vector $\hat \bbeta_{\EN}$ is mean risk consistent, that is,
\begin{eqnarray*}
\left | \sqrt{n} \left( \bar{\bx}_A \right)^T  \left(  \hat \bbeta_{\EN}  - \bbeta \right) \right| & \leq &   \sqrt{n}  \shortnorm{ \bar{\bx}_A }_\infty  \shortnorm{  \hat \bbeta_{\EN}  - \bbeta }_1   \nonumber \\
& = &   \sqrt{n}  O_p \left( \sqrt{ \frac{\log p}{n} }  \right)  o_p \left(  \frac{1}{\sqrt{\log p}} \right) = o_p(1). 
\end{eqnarray*}

\end{proof}

\subsection{Proof of Theorem \ref{conservative_variance}}

The following Lemma is used to control the degrees of freedom of the Elastic Net estimators. We will prove it in the later section.
\begin{lem}
\label{lem:num-of-selcted-varibles}
Under assumptions of Theorem \ref{conservative_variance}, there exists a constant $C$, such that the following holds with probability going to 1:
\begin{equation}
\hat s^{(a)} \leq C s; \ \ \hat s^{(b)} \leq C s.
\end{equation}
\end{lem}

\begin{proof}

To prove Theorem~\ref{conservative_variance}, it is enough to show that
\begin{equation}
\label{conver_vara}
\hat \sigma^2_{e^{(a)}} \stackrel{p}{\rightarrow}  \mathop {\lim}\limits_{n \rightarrow \infty }   \sigma^2_{e^{(a)}},
\end{equation}
\begin{equation}
\label{conver_varb}
\hat \sigma^2_{e^{(b)}} \stackrel{p}{\rightarrow}  \mathop {\lim}\limits_{n \rightarrow \infty } \sigma^2_{e^{(b)}}.
\end{equation}

Similar to \cite{bloniarz2015lasso}, we will only prove the statement \eqref{conver_vara} and omit the proof of the statement \eqref{conver_varb} since it is very similar.

Recall the decomposition of potential outcome \eqref{eqn:a}, we have
\[  a_i = \bx_i^T \bbeta^{(a)} + e^{(a)}_i;  \ \  \bar{a}_A = ( \bar{\bx}_A )^T \bbeta^{(a)} +  \bar{e}_A^{(a)}, \]
and hence
\begin{eqnarray}
& & \hat \sigma^2_{e^{(a)}} \nonumber \\
& = & \frac{1}{n_A - df^{(a)} } \sum_{i \in A} \left( a_i - \bar{a}_A - ( \bx_i - \bar{\bx}_A )^T \hat \bbeta^{(a)}_{\EN} \right)^2  \nonumber \\
& = & \frac{n_A}{n_A - df^{(a)} } \left\{  \frac{1}{n_A} \sum_{i \in A} \left( e^{(a)}_i  -  \bar{e}_A^{(a)} \right)^2 + \frac{1}{n_A} \sum_{i \in A} \left( ( \bx_i - \bar{\bx}_A )^T ( \bbeta^{(a)} - \hat \bbeta^{(a)}_{\EN}  ) \right)^2 \right\}\nonumber \\
&& + \frac{n_A}{n_A - df^{(a)} } \left\{ \frac{2}{n_A} \sum_{i \in A}  ( e^{(a)}_i  -  \bar{e}_A^{(a)} )( \bx_i - \bar{\bx}_A )^T ( \bbeta^{(a)} - \hat \bbeta^{(a)}_{\EN}  ) \right\}. \nonumber
\end{eqnarray}
By the fourth moment condition on the approximation error $e^{(a)}$ (see \eqref{cond:errmoment}), and applying Lemma \ref{lem:subsamplemean} we get
\[  \frac{1}{n_A} \sum_{i \in A} ( e^{(a)}_i)^2 \stackrel{p}{\rightarrow} \mathop {\lim}\limits_{n \rightarrow \infty } \sigma^2_{e^{(a)}}; \ \ \bar{e}_A^{(a)} \stackrel{p}{\rightarrow} \mathop {\lim}\limits_{n \rightarrow \infty }  \bar e^{(a)} =0. \]
Therefore,
\begin{equation}
\label{part1}
\frac{1}{n_A} \sum_{i \in A} \left( e^{(a)}_i   -  \bar{e}_A^{(a)} \right)^2 = \frac{1}{n_A} \sum_{i \in A} \left(( e^{(a)}_i)^2 - (\bar{e}_A^{(a)})^2 \right) \stackrel{p}{\rightarrow} \mathop {\lim}\limits_{n \rightarrow \infty } \sigma^2_{e^{(a)}}.
\end{equation}
It is easy to show that
\begin{eqnarray}
\label{part2}
& & \frac{1}{n_A} \sum_{i \in A} \left( (\bx_i - \bar{\bx}_A)^T ( \bbeta^{(a)} - \hat \bbeta^{(a)}_{\EN}  ) \right)^2  \nonumber \\
& = & \left( \bbeta^{(a)} - \hat \bbeta^{(a)}_{\EN}  \right)^T \left[ \frac{1}{n_A} \sum_{i \in A} (\bx_i - \bar{\bx}_A)(\bx_i - \bar{\bx}_A)^T \right]  \left( \bbeta^{(a)} - \hat \bbeta^{(a)}_{\EN}  \right) \nonumber \\
& \leq & || \bbeta^{(a)} - \hat \bbeta^{(a)}_{\EN} ||_1^2  \cdot || \frac{1}{n_A} \sum_{i \in A} (\bx_i - \bar{\bx}_A)(\bx_i - \bar{\bx}_A)^T  ||_\infty  \stackrel{p}{\rightarrow} 0.
\end{eqnarray}
By Cauchy-Schwarz inequality,
\begin{eqnarray}
\label{part3}
& & \left| \frac{1}{n_A} \sum_{i \in A}  ( e^{(a)}_i  -  \bar{e}_A^{(a)} )(\bx_i - \bar{\bx}_A)^T ( \bbeta^{(a)} - \hat \bbeta^{(a)}_{\EN}  )  \right| \nonumber \\
& \leq &  \left[ \frac{1}{n_A} \sum_{i \in A}  \left( e^{(a)}_i  -  \bar{e}_A^{(a)} \right)^2 \right]^{\frac{1}{2}} \cdot \left[   \frac{1}{n_A} \sum_{i \in A} \left( (\bx_i - \bar{\bx}_A)^T ( \bbeta^{(a)} - \hat \bbeta^{(a)}_{\EN}  ) \right)^2 \right]^{\frac{1}{2}}, \nonumber
\end{eqnarray}
which converges to $0$ in probability because of \eqref{part1} and \eqref{part2}.

By Lemma~\ref{lem:num-of-selcted-varibles}, we have
\begin{equation}
\label{part4}
\frac{n_A}{n_A - df^{(a)} } = \frac{n_A}{n_A - \hat s^{(a)} -1 } \stackrel{p}{\rightarrow} 1.
\end{equation}
Combining \eqref{part1}, \eqref{part2} and \eqref{part4}, we conclude that
\[ \hat \sigma^2_{e^{(a)}}  \stackrel{p}{\rightarrow} \mathop {\lim}\limits_{n \rightarrow \infty }   \sigma^2_{e^{(a)}}.\]

\end{proof}


\subsubsection{Proof of Theorem~\ref{thm:ada}}

\begin{proof}
We will make a connection between the Adaptive Lasso estimator and the Lasso estimator computed from the rescaled covariates $\bv \bx_i$, then the conclusion follows directly from the proof in paper \cite{bloniarz2015lasso}. By the definition of Adaptive Lasso estimator,
\begin{eqnarray}
\hat \bbeta_{\ada}  &=&   \argmin_{\bbeta} \frac{1}{2n_A} \sum_{i \in A} \left( a_i - \bar{a}_A - \left( \bx_i - \bar{\bx}_A \right)^T \bbeta \right)^2 +  \lambda_{a,1} \sum_{j=1}^{p}  w_{a,j} | \beta_j | \nonumber \\
& = &  \argmin_{\bbeta} \frac{1}{2n_A} \sum_{i \in A} \left( a_i - \bar{a}_A - \left( \bv \bx_i -  \bv \bar{ \bx}_A \right)^T \bw \bbeta \right)^2 +  \lambda_{a,1} \shortnorm{  \bw \bbeta }_1, \nonumber
\end{eqnarray}
where $\bv = \textnormal{diag}\{ 1/w_{a,1}, \cdots, 1/w_{a,p} \}$ and $\bw =  \textnormal{diag}\{ w_{a,1}, \cdots, w_{a,p} \} = \bv^{-1}$. Thus, $\bw  \hat \bbeta_{\ada} $ is the Lasso solution for the rescaled covariates $ \bv \bx_i $, that is,
\begin{equation}
\bw  \hat \bbeta_{\ada} = \argmin_{\bbeta} \frac{1}{2n_A} \sum_{i \in A} \left( a_i - \bar{a}_A - \left( \bv \bx_i -  \bv \bar{ \bx}_A \right)^T  \bbeta \right)^2 +  \lambda_{a,1} \shortnorm{   \bbeta }_1. \nonumber
\end{equation}
Then, under the assumptions of Theorem~\ref{thm:ada}, and by the proof in paper \cite{bloniarz2015lasso}, it holds that
\[  \shortnorm{ \bw  \hat \bbeta_{\ada}  - \bw \bbeta  }_1  = o_p \left(  \frac{1}{\sqrt{\log p}} \right).    \]
Therefore,
\begin{eqnarray*}
\left | \sqrt{n} \left( \bar{\bx}_A \right)^T  \left(  \hat \bbeta_{\ada}  - \bbeta \right) \right|  & = & \left | \sqrt{n} \left( \bv \bar{\bx}_A \right)^T  \left( \bw \hat \bbeta_{\ada}  - \bw \bbeta \right) \right|   \nonumber \\
& \leq &   \sqrt{n}  \shortnorm{ \bv \bar{\bx}_A }_\infty  \shortnorm{  \bw \hat \bbeta_{\ada}  - \bw \bbeta }_1   \nonumber \\
& = &   \sqrt{n}  O_p \left( \sqrt{ \frac{\log p}{n} }  \right)  o_p \left(  \frac{1}{\sqrt{\log p}} \right)= o_p(1), 
\end{eqnarray*}
where the last equality is due to \eqref{lem:xterm-ada} of Lemma~\ref{lem:concentrationofterms} for the rescaled covariates $\bv \bx_i$.

\end{proof}

\subsubsection{Proof of Theorem~\ref{thm:variance-ada}}
\begin{proof}
Again, the proof is almost the same as the proof of Theorem~\ref{conservative_variance} (replacing the Elastic Net adjusted vector by the Adaptive Lasso adjusted vector), so we omit it here.
\end{proof}

\subsection{Proof of Lemmas}

\subsubsection{Proof of Lemma~\ref{lem:eigen-ridge}}

\begin{proof}
It is enough to show that
\[ P \left(  \lambda_{\min} \left(  \hat \Sigma_A + \lambda_{a,2} I  \right) \geq \Lambda_{\min}/ 2\right)  \rightarrow 1 .\]
By the definition of the smallest eigenvalue, we have
\[  \lambda_{\min} \left(  \hat \Sigma_A + \lambda_{a,2} I  \right)  = \inf_{||u||_2 = 1} u^T   \left(  \hat \Sigma_A + \lambda_{a,2} I  \right) u .   \]
For any $u \in R^p , \ || u ||_2 = 1$,
\begin{eqnarray}
&& u^T   \left(  \hat \Sigma_A + \lambda_{a,2} I  \right) u  - u^T   \left( \Sigma + \lambda_{a,2} I  \right) u \nonumber \\
& = & u^T   \left(  \hat \Sigma_A - \Sigma  \right) u  \leq  \shortnorm{ \hat \Sigma_A - \Sigma }_\infty || u ||_1^2 \leq \shortnorm{ \hat \Sigma_A - \Sigma }_\infty p || u ||_2^2 \nonumber \\
& = & O_p     \left( p \sqrt{ \frac{\log p}{n} }  \right) = o_p(1),
\end{eqnarray}
where the last but second equality is due to \eqref{cov-mat-part1} of Lemma~\ref{lem:concentrationofterms}. The conclusion follows directly from  assumption~\ref{cond:smallest-eigen}, 
\begin{equation*}
\lambda_{\min}\left(  \Sigma + \lambda_{a,2} I    \right)  \geq \Lambda_{\min} > 0.
\end{equation*}

\end{proof}

\subsubsection{Proof of Lemma~\ref{lem:eventL}}

\begin{proof}
Recall that,
\[ 
\mathcal{L} = \left\{ \shortnorm{\frac{1}{n_A} \sum_{i \in A} ( \bx_i - \bar{\bx}_A ) ( e_i - \bar{e}_A ) }_\infty  + || \lambda_{a,2} \left( I -  \hat \Sigma^{(a)}_{\EN}  \right) \bbeta ||_\infty \leq \eta \lambda_{a,1}  \right\}.
\]
By \eqref{xeterm} of Lemma~\ref{lem:concentrationofterms}
\begin{equation}
\label{lasso_part}
P \left( \shortnorm{\frac{1}{n_A} \sum_{i \in A} ( \bx_i - \bar{\bx}_A ) ( e_i - \bar{e}_A ) }_\infty \leq  \frac{2(1+\varsigma)L^{1/2}}{p_A} \sqrt{ \frac{2\log p}{n} } + \delta_n \right) \rightarrow 1.
\end{equation}
Moreover, simple algebra gives
\begin{eqnarray}
 \shortnorm{ \lambda_{a,2} \left( I -  \hat \Sigma^{(a)}_{\EN}  \right) \bbeta }_\infty  & = & \frac{ \lambda_{a,2} }{ 1 + \lambda_{a,2} } \shortnorm{  \left( I - \hat \Sigma_A \right)  \bbeta }_\infty \nonumber \\
& \leq & \frac{ \lambda_{a,2} }{ 1 + \lambda_{a,2} }   \left( \shortnorm{  I -   \Sigma }_\infty  + \shortnorm{ \hat \Sigma_A  - \Sigma  }_\infty \right) \shortnorm{ \bbeta}_1 , \nonumber
\end{eqnarray} 
where $\Sigma = X^TX/n = \sum_{i=1}^{n} \bx_i \bx_i^T/n$. By \eqref{cov-mat-part1} of Lemma~\ref{lem:concentrationofterms} and the scaling assumption \eqref{cond:s-scaling}, we have
\[ \shortnorm{ \hat \Sigma_A - \Sigma  }_\infty = o_p(1).    \]
By the fourth moment conditions on the covariates \eqref{cond:xmoment}, it is easy to show that
\[ \shortnorm{  I -   \Sigma }_\infty  \leq  L^{1/2} + 1 . \]
Therefore,
\begin{equation}
\label{en_part}
\shortnorm{ \lambda_{a,2} \left( I -  \hat \Sigma^{(a)}_{\EN}  \right) \bbeta }_\infty  \leq  \frac{ \lambda_{a,2} }{ 1 + \lambda_{a,2} } \left( L^{1/2} + 1 \right) || \bbeta ||_1 \left( 1 + o_p(1) \right).
\end{equation}
Combing \eqref{lasso_part}, \eqref{en_part} and assumption~\ref{last-last-cond} on $\lambda_{a,1}$, 
\begin{equation}
\lambda_{a,1} \in   \left( \frac{1}{\eta}, M \right] \times \left( \frac{2(1+\varsigma)L^{1/2}}{p_A} \sqrt{ \frac{2\log p}{n} } + \delta_n + \frac{ \lambda_{a,2} } { 1 + \lambda_{a,2} } ( L^{1/2} + 1 ) || \bbeta^{(a)}  ||_1 \right), \nonumber
\end{equation}
we have $P(\mathcal{L} ) \rightarrow 1 $.

\end{proof}

\subsubsection{Proof of Lemma~\ref{lem:eventM}}

\begin{proof}
Recall that,
\[   
\hat \Sigma^{(a)}_{\EN} = (1+\lambda_{a,2})^{-1}\left( \hat \Sigma_A + \lambda_{a,2}I\right), \quad \Sigma^{(a)}_{\EN} = (1+\lambda_{a,2})^{-1}\left( \Sigma + \lambda_{a,2}I\right) ,
\]
then,
\begin{eqnarray}
\shortnorm{ \hat \Sigma^{(a)}_{\EN}  - \Sigma^{(a)}_{\EN} }_\infty & = & (1+\lambda_{a,2})^{-1} \shortnorm{ \hat \Sigma_A -  \Sigma }_\infty 
 \leq  \shortnorm{ \hat \Sigma_A -  \Sigma }_\infty .
\end{eqnarray}
The conclusion follows directly from \eqref{cov-mat-part1} of Lemma~\ref{lem:concentrationofterms}.

\end{proof}

\subsubsection{Proof of Lemma~\ref{lem:num-of-selcted-varibles}}

\begin{proof}
In the proof of Theorem~\ref{mean risk consistent}, recalling $ - \bh =   \bbeta - \hat \bbeta_{\EN} $,
we have shown that on $\mathcal{L}$,
\begin{align} \label{on-event-L}
\bh^T \hat \Sigma^{(a)}_{\EN}\bh  &\leq
\lambda_{a,1} \left( \shortnorm{\bbeta}_1 - \shortnorm{\hat\bbeta_{\EN}}_1\right) + \eta \lambda_{a,1} \shortnorm{\bh}_1\notag \leq \lambda_{a,1} \left(1 +  \eta \right) || \bbeta - \hat \bbeta_{\EN} ||_1.
\end{align}
By KKT condition, we have shown that if $\hat \beta_{\EN,j} \neq 0$:
\begin{eqnarray}
\left |  \hat \Sigma^{(a)}_{\EN}   \left( \bbeta - \hat \bbeta_{\EN}  \right) + \frac{1}{n_A} \sum_{i\in A} (\bx_i - \bar{\bx}_A)( e_i - \bar{e}_A )  - \lambda_{a,2} \left( I -  \hat \Sigma_{\EN}  \right) \bbeta \right|_j  = \lambda_{a,1} , \nonumber 
\end{eqnarray}
Let $\Delta = \left| \hat \Sigma^{(a)}_{\EN}  ( \bbeta - \hat \bbeta_{\EN} ) \right| $, then on $\mathcal{L}$, we have if $\hat \beta_{\EN,j} \neq 0$
\begin{equation}
\label{lowerbound}
 \Delta_j   \geq (1 - \eta) \lambda_{a,1}.
\end{equation}
Therefore,
\begin{equation}
\label{upperbound}
\Delta^T \Delta =  \sum_{j=1}^{p} \Delta_j^2 \geq   \sum_{j: \hat \beta_{\EN,j} \neq 0} \Delta_j^2 \geq  (1 - \eta)^2\lambda_{a,1}^2 \hat s.
\end{equation}
On the other hand,
\begin{eqnarray}
\label{lowerbound}
\Delta^T \Delta & = & \bh^T \hat \Sigma^{(a)}_{\EN}  \hat \Sigma^{(a)}_{\EN} \bh  \leq  \lambda_{\max} \left(  \hat \Sigma^{(a)}_{\EN}  \right)  \bh^T \hat \Sigma^{(a)}_{\EN}  \bh  \nonumber \\
& \leq & \max\{  (n/n_A) \Lambda_{\max}, 1  \} \lambda_{a,1} \left(1 +  \eta \right) || \bbeta - \hat \bbeta_{\EN} ||_1,
\end{eqnarray}
where the last inequality is because of assumption~\ref{cond:largest} and the fact that
\[   \lambda_{\max} \left(  \hat \Sigma^{(a)}_{\EN}  \right)  \leq \frac{ \lambda_{\max}\left( (n/n_A) \Sigma \right) + \lambda_{a,2} }{ 1 + \lambda_{a,2} } \leq \max\{ (n/n_A) \Lambda_{\max}, 1   \} . \]
Combining \eqref{upperbound}, \eqref{lowerbound} and with probability going to $1$, $|| \bbeta - \hat \bbeta_{\EN} ||_1 \leq C s ( 1  + \eta ) \lambda_{a,1}$, where $C$ is a constant, we conclude that with probability going to $1$,
\begin{equation*}
\begin{split}
 \hat s & \leq \frac{1}{(1-\eta)^2} \frac{1}{\lambda_{a,1}^2} \max\{ (n/n_A) \Lambda_{\max}, 1  \}  \lambda_{a,1} ( 1  + \eta ) C s ( 1  + \eta ) \lambda_{a,1} \\
 & \leq \frac{C( 1+\eta )^2 \max\{ (n/n_A) \Lambda_{\max}, 1  \}}{(1-\eta)^2} s.
\end{split}
\end{equation*}

\end{proof}

\section{Tables}\label{app:tables}

\begin{table}[ht]
\centering
 \caption{\label{tab:mse1} Bias$^2$, Variance, MSE, coverage probability (Coverage) and mean interval length (Length) of different ACE estimators for Example 1.}
\begin{tabular*}{\hsize}{@{\extracolsep{\fill}}lcccccc}
 & & &  $(p, n_A)$ & &  \\
  \cline{2-7}
 Method  & (50, 80) & (50, 100) & (50, 120) & (500, 80) & (500, 100) & (500, 120) \\
  \hline \\
    & & & Bias$^2 \times 1000$ & & \\ \hline

unadjust & 0.06(0.75) & 0.50(0.99) & \textbf{ 0.00(0.45) }& 0.26(0.83) & 0.15(0.69) & 0.05(0.53) \\ 
  OLS & 0.10(0.47) & 1.95(1.34) & 0.62(0.90) & - & - & - \\ 
  Lasso & 0.01(0.22) & 0.42(0.46) & 0.09(0.28) & 0.05(0.23) & 0.50(0.51) & 0.05(0.23) \\ 
  EN & \textbf{0.01(0.18)} & \textbf{0.32(0.42)} & 0.12(0.30) & \textbf{0.01(0.17)} & 0.38(0.48) & 0.04(0.19) \\ 
  naiveEN  & 0.04(0.21) & 0.42(0.45) & 0.15(0.34) & 0.02(0.19) & 0.45(0.56) & 0.04(0.25) \\ 
  Ada & 0.17(0.35) & 0.44(0.51) & 0.13(0.33) & 0.05(0.21) & 0.49(0.54) & 0.03(0.21) \\ 
  Ridge & 0.14(0.33) & 0.75(0.58) & 0.28(0.39) & 0.56(0.80) & \textbf{0.00(0.28)} & \textbf{0.00(0.30)} \\ 
    \hline \\
    & & & Var $\times 1000$ & & \\ \hline
    
  unadjust & 479(22) & 410(18) & 386(17) & 415(19) & 364(14) & 358(15) \\ 
  OLS & 256(11) & 203(9) & 274(12) & - & - & - \\ 
  Lasso & 132(6) & 115(5) & 129(6) & 119(6) & 119(5) & 127(5) \\ 
  EN & \textbf{121(5)} & \textbf{109(5)} & \textbf{122(6)} & \textbf{105(5)} & \textbf{105(4) }& \textbf{114(5)} \\ 
  naiveEN  & 129(6) & 112(5) & 127(6) & 119(6) & 119(5) & 130(6) \\ 
  Ada & 129(6) & 115(5) & 130(6) & 117(5) & 115(5) & 122(5) \\ 
  Ridge & 127(5) & 111(5) & 124(6) & 238(11) & 212(10) & 222(9) \\ 
    \hline \\
    & & & MSE $\times 1000$  & & \\ \hline
    
  unadjust & 479(22) & 410(18) & 386(17) & 415(18) & 364(15) & 358(15) \\ 
  OLS & 257(11) & 205(9) & 274(12) & - & - & - \\ 
  Lasso & 132(6) & 115(5) & 129(6) & 119(5) & 119(5) & 127(5) \\ 
  EN & \textbf{121(6)} & \textbf{109(5)} & \textbf{122(6)} & \textbf{105(5)} & \textbf{105(5)} & \textbf{114(5)} \\ 
  naiveEN  & 129(6) & 112(5) & 127(6) & 119(5) & 120(5) & 130(6) \\ 
  Ada & 129(6) & 116(5) & 130(6) & 117(5) & 116(5) & 122(5) \\ 
  Ridge & 127(6) & 112(5) & 125(6) & 238(12) & 212(9) & 222(9) \\ 
    \hline \\
    & & & Coverage & & \\ \hline

  unadjust & 96.9 & 97.8 & 97.1 & 98 & 98.3 & 98.6 \\ 
  OLS & 91.2 & 95.0 & 90.4 & - & - & - \\ 
  Lasso & 98.3 & 98.6 & 98.2 & 98.5 & 97.8 & 98.0 \\ 
  EN & 98.4 & 98.6 & 99.0 & 99.4 & 98.8 & 98.9 \\ 
  naiveEN  & 98.6 & 98.9 & 99.0 & 98.6 & 98.0 & 98.0 \\ 
  Ada & 97.7 & 98.1 & 97.9 & 97.7 & 98.3 & 98.6 \\ 
  Ridge & 96.9 & 98.3 & 97.8 & 89.4 & 92.1 & 91.3 \\ 
    \hline \\
    & & & Length & & \\ \hline
    
  unadjust & 3.04 & 2.87 & 2.8 & 3.03 & 2.85 & 2.78 \\ 
  OLS & 1.74 & 1.73 & 1.77 & - & - & - \\ 
  Lasso & 1.73 & 1.69 & 1.74 & 1.68 & 1.65 & 1.70 \\ 
  EN & 1.74 & 1.71 & 1.75 & 1.74 & 1.70 & 1.75 \\ 
  naiveEN  & 1.82 & 1.77 & 1.82 & 1.76 & 1.74 & 1.81 \\ 
  Ada & 1.67 & 1.65 & 1.69 & 1.67 & 1.65 & 1.69 \\ 
  Ridge & \textbf{1.56} & \textbf{1.56} & \textbf{1.59} & \textbf{1.62} & \textbf{1.60} & \textbf{1.62} \\ 
   \hline
\end{tabular*}
The numbers in parentheses are the corresponding standard errors estimated by using the bootstrap with $B=500$ resamples of the ATE estimates.
\end{table}

\begin{table}[ht]
\centering
 \caption{\label{tab:mse2} Bias$^2$, Variance, MSE, coverage probability (Coverage) and mean interval length (Length) of different ACE estimators for Example 2.}
\begin{tabular*}{\hsize}{@{\extracolsep{\fill}}lcccccc}
 & & &  $(p, n_A)$ & &  \\
  \cline{2-7}
 Method  & (50, 80) & (50, 100) & (50, 120) & (500, 80) & (500, 100) & (500, 120) \\
  \hline \\
    & & & Bias$^2 \times 1000$ & & \\ \hline
    
 unadjust & \textbf{0.00(0.73)} & \textbf{0.01(0.62)} & \textbf{0.54(1.56)} & 0.57(1.05) & 0.13(0.77) & \textbf{0.29(0.94)} \\ 
  OLS & 1.30(1.29) & 0.06(0.39) & 0.71(0.93) & - & - & - \\ 
  Lasso & 0.78(0.68) & 0.09(0.27) & 1.07(0.74) & 0.03(0.25) & 0.01(0.21) & 0.48(0.58) \\ 
  EN & 0.81(0.64) & 0.20(0.41) & 0.75(0.67) & 0.07(0.30) & \textbf{0.00(0.2)} & 0.34(0.46) \\ 
  naiveEN & 0.65(0.61) & 0.08(0.26) & 0.83(0.66) & 0.03(0.26) & 0.02(0.22) & 0.50(0.58) \\ 
  Ada & 1.06(0.81) & 0.13(0.31) & 1.42(0.91) & 0.01(0.25) & 0.04(0.25) & 0.51(0.62) \\ 
  Ridge & 0.56(0.55) & 0.17(0.33) & 0.76(0.58) & \textbf{0.00(0.35)} & 0.01(0.31) & 0.38(0.76) \\ 
      \hline \\
    & & & Var $\times 1000$ & & \\ \hline

  unadjust & 484(22) & 524(22) & 596(27) & 365(15) & 433(20) & 490(21) \\ 
  OLS & 280(14) & 209(9) & 267(14) & - & - & - \\ 
  Lasso & 132(6) & 126(5) & 128(6) & 150(6) & 142(6) & 151(7) \\ 
  EN & \textbf{126(6)} & \textbf{118(5)} & \textbf{116(5)} & \textbf{139(6)} & \textbf{127(5)} & \textbf{134(6)} \\ 
  naiveEN & 132(6) & 124(5) & 122(5) & 151(6) & 144(6) & 152(7) \\ 
  Ada & 133(5) & 130(6) & 129(6) & 150(6) & 138(6) & 145(7) \\ 
  Ridge & 129(6) & 125(5) & 124(5) & 234(9) & 255(11) & 295(13) \\ 
    \hline \\
    & & & MSE $\times 1000$  & & \\ \hline  
  
  unadjust & 484(22) & 524(23) & 597(25) & 365(15) & 433(18) & 491(21) \\ 
  OLS & 281(13) & 209(10) & 267(13) & - & - & - \\ 
  Lasso & 133(6) & 126(5) & 129(6) & 150(6) & 142(6) & 151(7) \\ 
  EN & \textbf{127(6)} & \textbf{118(5)} & \textbf{117(5)} & \textbf{139(6)} & \textbf{127(5)} & \textbf{134(6)} \\ 
  naiveEN & 133(6) & 124(5) & 123(5) & 151(6) & 144(6) & 153(7) \\ 
  Ada & 134(6) & 130(6) & 130(6) & 150(6) & 138(6) & 146(6) \\ 
  Ridge & 130(5) & 125(5) & 125(6) & 234(10) & 255(11) & 295(13) \\ 
    \hline \\
    & & & Coverage & & \\ \hline  
  
  unadjust & 97.4 & 97.5 & 97.3 & 98.9 & 97.7 & 97.3 \\ 
  OLS & 91.8 & 94.6 & 93.1 & - & - & - \\ 
  Lasso & 99.3 & 99.2 & 99.1 & 97.4 & 98.0 & 96.9 \\ 
  EN & 99.3 & 99.5 & 99.6 & 98.5 & 99.1 & 98.5 \\ 
  naiveEN & 99.2 & 99.6 & 99.6 & 96.4 & 98.0 & 96.3 \\ 
  Ada & 98.6 & 98.7 & 98.7 & 98.0 & 98.1 & 98.2 \\ 
  Ridge & 98.1 & 98.4 & 98.0 & 90.5 & 89.9 & 88.0 \\ 
    \hline \\
    & & & Length & & \\ \hline 
 
  unadjust & 3.18 & 3.24 & 3.44 & 2.92 & 2.97 & 3.14 \\ 
  OLS & 1.84 & 1.81 & 1.84 & - & - & - \\ 
  Lasso & 1.88 & 1.84 & 1.89 & 1.76 & 1.74 & 1.75 \\ 
  EN & 1.89 & 1.86 & 1.89 & 1.86 & 1.83 & 1.86 \\ 
  naiveEN & 1.94 & 1.90 & 1.96 & 1.88 & 1.88 & 1.89 \\ 
  Ada & 1.83 & 1.79 & 1.82 & 1.80 & 1.77 & 1.79 \\ 
  Ridge & \textbf{1.69} & \textbf{1.66} & \textbf{1.66} & \textbf{1.64} & \textbf{1.63} & \textbf{1.65} \\ 
   \hline
\end{tabular*}
The numbers in parentheses are the corresponding standard errors estimated by using the bootstrap with $B=500$ resamples of the ATE estimates.
\end{table}

\begin{table}[ht]
\centering
 \caption{\label{tab:mse3} Bias$^2$, Variance, MSE, coverage probability (Coverage) and mean interval length (Length) of different ACE estimators for Example 3.}
\begin{tabular*}{\hsize}{@{\extracolsep{\fill}}lcccccc}
 & & &  $(p, n_A)$ & &  \\
  \cline{2-7}
 Method  & (50, 80) & (50, 100) & (50, 120) & (500, 80) & (500, 100) & (500, 120) \\
  \hline \\
    & & & Bias$^2 \times 1000$ & & \\ \hline

  unadjust & \textbf{0.00(0.69)} & 0.18(0.79) & 0.39(0.96) & 0.19(0.72) & 0.27(0.79) & 0.42(0.84) \\ 
  OLS & 0.97(1.17) & 0.01(0.33) & 0.04(0.45) & - & - & - \\ 
  Lasso & 0.40(0.57) & \textbf{0.00(0.2)} & 0.03(0.25) & 0.01(0.21) & 0.05(0.25) & 0.16(0.34) \\ 
  EN & 0.26(0.44) & 0.01(0.19) & 0.06(0.25) & \textbf{0.00(0.18)} & 0.04(0.23) & \textbf{0.06(0.25)} \\ 
  naiveEN & 0.30(0.43) & 0.01(0.17) & 0.02(0.19) & 0.01(0.22) & 0.04(0.26) & 0.12(0.31) \\ 
  Ada & 0.59(0.59) & 0.02(0.21) & \textbf{0.01(0.20) }& 0.00(0.22) & 0.02(0.26) & 0.12(0.31) \\ 
  Ridge & 0.14(0.35) & 0.03(0.22) & 0.06(0.25) & 0.05(0.25) & \textbf{0.02(0.19)} & 0.11(0.36) \\ 
      \hline \\
    & & & Var $\times 1000$ & & \\ \hline
      
   unadjust & 516(24) & 444(21) & 400(18) & 395(18) & 373(16) & 363(15) \\ 
  OLS & 309(13) & 243(12) & 264(11) & - & - & - \\ 
  Lasso & 146(7) & 130(6) & 134(6) & 144(6) & 132(6) & 138(6) \\ 
  EN & \textbf{138(6)} & \textbf{125(6)} & \textbf{126(5)} & 142(7) & 130(6) & 137(6) \\ 
  naiveEN & 140(6) & 129(6) & 130(6) & 140(6) & 130(6) & 136(5) \\ 
  Ada & 148(7) & 128(6) & 129(6) & 156(7) & 139(6) & 146(6) \\ 
  Ridge & 138(6) & 129(6) & 127(5) & \textbf{135(6)} & \textbf{124(6)} & \textbf{130(5)} \\ 
    \hline \\
    & & & MSE $\times 1000$  & & \\ \hline    
  
   unadjust & 516(25) & 444(20) & 400(19) & 396(18) & 374(17) & 364(15) \\ 
  OLS & 310(13) & 243(11) & 264(12) & - & - & - \\ 
  Lasso & 147(7) & 130(6) & 134(6) & 144(6) & 132(6) & 138(5) \\ 
  EN & \textbf{138(6)} & \textbf{125(6)} & \textbf{126(6) }& 142(6) & 130(6) & 137(6) \\ 
  naiveEN & 140(6) & 129(7) & 130(6) & 140(6) & 130(6) & 136(6) \\ 
  Ada & 149(7) & 128(6) & 129(6) & 156(7) & 139(6) & 146(6) \\ 
  Ridge & 139(6) & 129(6) & 128(6) & \textbf{135(6)} & \textbf{124(6)} & \textbf{130(5)} \\ 
    \hline \\
    & & & Coverage & & \\ \hline  
      
   unadjust & 96.8 & 97.2 & 97.7 & 97.8 & 97.6 & 97.5 \\ 
  OLS & 90.8 & 93.9 & 94.3 & - & - & - \\ 
  Lasso & 98 & 98.3 & 98.8 & 96.8 & 97.7 & 97.9 \\ 
  EN & 99 & 98.8 & 99.4 & 97.3 & 97.7 & 97.7 \\ 
  naiveEN & 99.5 & 99.3 & 99.6 & 78.0 & 82.1 & 84.5 \\ 
  Ada & 97.4 & 98.2 & 98.7 & 96.2 & 97.4 & 97.1 \\ 
  Ridge & 98.1 & 98.3 & 98.5 & 98.0 & 97.8 & 97.9 \\ 
    \hline \\
    & & & Length & & \\ \hline   
  
   unadjust & 3.19 & 2.99 & 2.91 & 2.92 & 2.76 & 2.72 \\ 
  OLS & 1.92 & 1.88 & 1.92 & - & - & - \\ 
  Lasso & 1.90 & 1.85 & 1.87 & 1.72 & 1.69 & 1.72 \\ 
  EN & 1.99 & 1.92 & 1.95 & 1.98 & 1.90 & 1.87 \\ 
  naiveEN & 2.21 & 2.10 & 2.17 & 1.98 & 1.64 & 1.52 \\ 
  Ada & 1.82 & 1.79 & 1.83 & 1.69 & 1.69 & 1.72 \\ 
  Ridge & \textbf{1.81} & \textbf{1.77} & \textbf{1.79} & \textbf{1.67} & \textbf{1.64} & \textbf{1.66} \\ 
   \hline
\end{tabular*}
The numbers in parentheses are the corresponding standard errors estimated by using the bootstrap with $B=500$ resamples of the ATE estimates.
\end{table}

\begin{table}[ht]
\centering
 \caption{\label{tab:mse4} Bias$^2$, Variance, MSE, coverage probability (Coverage) and mean interval length (Length) of different ACE estimators for Example 4.}
\begin{tabular*}{\hsize}{@{\extracolsep{\fill}}lcccccc}
 & & &  $(p, n_A)$ & &  \\
  \cline{2-7}
 Method  & (50, 80) & (50, 100) & (50, 120) & (500, 80) & (500, 100) & (500, 120) \\
  \hline \\
    & & & Bias$^2 \times 1000$ & & \\ \hline
    
  unadjust & 1.78(3.25) & 0.52(2.21) & 0.15(1.55) & 0.04(2.11) & 1.24(3.54) & 0.47(2.31) \\ 
  OLS & 1.89(1.03) & 0.01(0.15) & 0.10(0.28) & - & - & - \\ 
  Lasso & 0.31(0.36) & 0.02(0.12) & 0.02(0.15) & 0.01(0.46) & 0.88(1.42) & \textbf{0.01(0.88)} \\ 
  EN & 0.37(0.35) & 0.00(0.11) & 0.03(0.14) & 0.01(0.39) & 0.51(1.03) & 0.05(0.96) \\ 
  naiveEN & 0.39(0.35) & \textbf{0.00(0.10)} & 0.05(0.16) & 0.08(0.53) & 0.64(1.27) & 0.02(0.85) \\ 
  Ada & 0.46(0.38) & 0.02(0.13) & \textbf{0.02(0.12)} & \textbf{0.00(0.42)} & \textbf{0.30(0.86)} & 0.04(0.95) \\ 
  Ridge & \textbf{0.31(0.34)} & 0.01(0.11) & 0.02(0.13) & \textbf{0.00(1)} & 0.68(1.56) & 0.16(1.52) \\ 
      \hline \\
    & & & Var $\times 1000$ & & \\ \hline
      
   unadjust & 1127(55) & 1022(45) & 1030(46) & 1407(61) & 1283(58) & 1361(59) \\ 
  OLS & 143(6) & 107(5) & 131(6) & - & - & - \\ 
  Lasso & 84(4) & 75(3) & 88(3) & 344(15) & 453(22) & 638(29) \\ 
  EN & \textbf{72(3)} & \textbf{69(3)} & \textbf{80(3)} & \textbf{300(13)} & \textbf{411(19)} & \textbf{608(26)} \\ 
  naiveEN & 77(4) & 71(3) & 84(4) & 310(13) & 414(20) & 602(27) \\ 
  Ada & 74(4) & 74(3) & 83(4) & 317(14) & 438(20) & 629(26) \\ 
  Ridge & 87(4) & 75(3) & 83(4) & 730(32) & 690(32) & 800(34) \\ 
    \hline \\
    & & & MSE $\times 1000$  & & \\ \hline  
      
   unadjust & 1129(53) & 1022(46) & 1030(47) & 1407(63) & 1284(59) & 1361(62) \\ 
  OLS & 145(5) & 107(5) & 131(7) & - & - & - \\ 
  Lasso & 84(4) & 75(3) & 88(4) & 344(15) & 454(22) & 638(29) \\ 
  EN & \textbf{72(3)} &\textbf{ 69(3)} & \textbf{80(3)} & \textbf{300(13) }& \textbf{412(18)} & \textbf{608(26)} \\ 
  naiveEN & 77(4) & 71(3) & 84(4) & 310(13) & 415(21) & 602(27) \\ 
  Ada & 75(3) & 74(3) & 83(4) & 317(15) & 438(20) & 630(28) \\ 
  Ridge & 87(4) & 75(3) & 84(4) & 730(33) & 691(31) & 801(35) \\ 
    \hline \\
    & & & Coverage & & \\ \hline  
      
   unadjust & 95.5 & 95.3 & 95.2 & 97.2 & 97.4 & 97.3 \\ 
  OLS & 90.3 & 93.2 & 92.7 & - & - & - \\ 
  Lasso & 96.9 & 97.6 & 97.0 & 94.6 & 89.5 & 87.4 \\ 
  EN & 98.4 & 98.9 & 98.1 & 84.8 & 78.5 & 71.2 \\ 
  naiveEN & 97.7 & 98.6 & 98.4 & 66.6 & 63.2 & 52.3 \\ 
  Ada & 97.1 & 97.1 & 96.9 & 97.4 & 95.6 & 92.8 \\ 
  Ridge & 90.8 & 93.5 & 91.7 & 80.6 & 81.3 & 76.8 \\ 
    \hline \\
    & & & Length & & \\ \hline 
      
   unadjust & 4.32 & 4.12 & 4.07 & 5.16 & 5.13 & 5.31 \\ 
  OLS & 1.27 & 1.23 & 1.24 & - & - & - \\ 
  Lasso & 1.33 & 1.28 & 1.32 & 2.49 & 2.67 & 2.98 \\ 
  EN & 1.34 & 1.31 & 1.35 & 2.11 & 2.30 & 2.41 \\ 
  naiveEN & 1.34 & 1.33 & 1.40 & \textbf{1.30} & \textbf{1.45} & \textbf{1.53} \\ 
  Ada & 1.24 & 1.23 & 1.27 & 2.89 & 3.10 & 3.39 \\ 
  Ridge & \textbf{1.03} & \textbf{1.02} & \textbf{1.02} & 2.24 & 2.22 & 2.20 \\ 
   \hline
\end{tabular*}
The numbers in parentheses are the corresponding standard errors estimated by using the bootstrap with $B=500$ resamples of the ATE estimates.
\end{table}

\end{document}